\documentclass[10pt]{amsart}
\usepackage{amsmath,amscd}
\usepackage{amssymb}

\usepackage[all]{xy}

\usepackage[T1]{fontenc}

\usepackage{times}

\usepackage[T1]{fontenc}
\usepackage{graphicx}
\usepackage[svgnames]{xcolor}
\usepackage{framed}

\newcommand*\openquote{\makebox(25,-10){\scalebox{3}{``}}}
\newcommand*\closequote{\makebox(25,-10){\scalebox{3}{''}}}
\colorlet{shadecolor}{White}
\makeatletter
\newif\if@right
\def\shadequote{\@righttrue\shadequote@i}
\def\shadequote@i{\begin{snugshade}\begin{quote}\openquote}
\def\endshadequote{%
  \if@right\hfill\fi\closequote\end{quote}\end{snugshade}}
\@namedef{shadequote*}{\@rightfalse\shadequote@i}
\@namedef{endshadequote*}{\endshadequote}
\makeatother

\usepackage{enumitem}
\newlist{steps}{enumerate}{1}
\setlist[steps, 1]{label = Step \arabic*:}

\newtheorem{thm}[equation]{Theorem}
\newtheorem*{thm*}{Theorem}
\newtheorem{cor}[equation]{Corollary}
\newtheorem{prop}[equation]{Proposition}
\newtheorem{lem}[equation]{Lemma}
\newtheorem{prob}[equation]{Problem}
\theoremstyle{definition}
\newtheorem{dfn}[equation]{Definition}

\newtheorem{fact}[equation]{}

\newcommand{\opn}{\operatorname}

\newcommand{\OpenLoop}{\mathcal{B}}
\newcommand{\ClosedLoop}{\mathcal{B}_F^{\opn{CL}}}
\newcommand{\Controller}{\mathcal{B}_F^{\opn{Ctrl}}}
\newcommand{\fdim}{\operatorname{fb.dim}}

\title{Pole placement for overdetermined 2D systems}
\author{Liran Shaul}
\author{Victor Vinnikov}

\address{Shaul: Department of Algebra, Faculty of Mathematics and Physics, Charles University in Prague, Sokolovsk\'a 83, 186 75 Praha, Czech Republic}

\email{shaul@karlin.mff.cuni.cz}

\address{Vinnikov: Department of Mathematics, Ben-Gurion University of the Negev, Beer Sheva 84105, Israel}

\email{vinnikov@math.bgu.ac.il}

\thanks{{\em 2010 Mathematics Subject Classification:}
primary: 93B55, secondary: 14H60, 47N70, 93B25, 93B27 }

\begin{document}

\begin{abstract}
We formulate and solve a pole placement problem by state feedback for overdetermined 2D systems modeled by commutative operator vessels. In this setting, the transfer function of the system is given by a meromorphic bundle map between two holomorphic vector bundles of finite rank over the normalization of a projective plane algebraic curve. The obstruction for a solution is given by an existence of a certain meromorphic bundle map on the input bundle. Reducing to the 1D case, this gives a functional obstruction which is equivalent to the classical pole placement theorem. Our result improves on, and gives a new approach to pole placement even in the classical case, and answers a question of Ball and Vinnikov.
\end{abstract}
\maketitle

\numberwithin{equation}{section}

\setcounter{section}{-1}
\section{Introduction}

An overdetermined 2D continuous-time time-invariant linear input-state-output system is a linear system described by the following system of equations:

\begin{equation}\label{eqn:ODSystem}
\Sigma:\left\{\begin{array}{c}
\frac{ \partial{x} } {\partial{t_1}}(t_1,t_2) = A_1 x(t_1,t_2)+ B_1 u(t_1,t_2)\\
\\
\frac{ \partial{x} } {\partial{t_2}}(t_1,t_2) = A_2 x(t_1,t_2)+ B_2 u(t_1,t_2)\\
\\
y(t_1,t_2) = C x(t_1,t_2) + D u(t_1,t_2).
\end{array}\right.
\end{equation}

Here, $u$, $x$ and $y$ represents the input, state, and output signals, respectively. The input space is denoted by $\mathcal{E}$, the state space by $\mathcal{H}$ and the output space by $\mathcal{E}_*$. All spaces are assumed to be Hilbert spaces over the complex numbers. The operators $A_1, A_2, B_1, B_2, C$ and $D$ act as follows:

\begin{equation}
\left.\begin{array}{c}
A_1,A_2:\mathcal{H} \rightarrow \mathcal{H} \\
B_1,B_2:\mathcal{E} \rightarrow \mathcal{H} \\
C:\mathcal{H} \rightarrow \mathcal{E}_*\\
D:\mathcal{E} \rightarrow \mathcal{E}_*
\end{array}\right.\nonumber
\end{equation}

Experience showed that a good model to study these kind of systems is a notion called a \textbf{Liv{\v{s}}ic-Kravitsky commutative two-operator vessel}. We recall the definition and most important properties of this model in Section 1 below. 
The purpose of this article is to initiate the development of a theory of state feedback for these kinds of systems. 
The next quote is taken from \cite[Page 14]{So}:

\begin{shadequote}
The Pole-Shifting Theorem is central to linear systems theory and is itself the starting point for more interesting analysis.
\end{shadequote}

Before discussing the specifics of the results of this paper, let us first put the pole-shifting theorem into context.
There are various different formulations of the pole-shifting theorem in control theory and linear algebra.
The \textit{Spectrum Assignment Problem} asks, 
given matrices $A \in M_{n\times n}(\mathbb{C})$, 
$B \in M_{n \times m}(\mathbb{C})$ if there exists 
$F \in M_{m \times n}(\mathbb{C})$ such that $A+BF$ has a given set of eigenvalues.  
This is true if and only if $(A,B)$ is controllable, equivalently, 
by the Hautus Lemma,
if the matrix $\left[ \begin{matrix} A-\lambda I & B \end{matrix}\right]$ has full rank for all $\lambda \in \mathbb{C}$.
A dual version studies the \textit{Spectrum Assignment Problem} for $A+GC$,
with $C \in M_{n \times m}(\mathbb{C})$ given, and asks for the existence of $G \in M_{m \times n}(\mathbb{C})$.
This happens exactly when $(C,A)$ is observable, equivalently, when 
$\left[ \begin{matrix} A-\lambda I \\ C \end{matrix}\right]$ has full rank for all $\lambda \in \mathbb{C}$.
The \textit{Spectrum Stabilization Problem} is a variation of the above which asks 
for which pairs $(A,B)$ can one stabilize $A+BF$, 
so that all of its eigenvalues will be in the left half-plane.
The Hautus Lemma for stabilizability states that this is the case if and only if 
$\left[ \begin{matrix} A-\lambda I & B \end{matrix}\right]$ has full rank for all 
$\lambda \in \mathbb{C}$ with $\Re(\lambda) \ge 0$. 
Again, a dual version exists which characterizes detectable pairs $(C,A)$.
An even more general linear algebra question asks, 
given $A \in M_{n\times n}(\mathbb{C})$, $B \in M_{n\times m}(\mathbb{C})$, $C \in M_{p\times n}(\mathbb{C})$
if there exists $K \in M_{m \times p}(\mathbb{C})$, 
such that $A+ BKC$ has a given set of eigenvalues.
See for example \cite{KY} for more details about these various questions and their solutions.

This paper deals with a generalization of the \textit{pole-shifting problem},
which is another variation of the above in control theory.
Consider an open loop continuous-time time-invariant linear input-state-output system of the form
\[
\Sigma: \left\{
\begin{matrix}
x'(t) = A\cdot x(t) + B \cdot u(t), & x(0) = x_0\\
y(t) = C\cdot x(t) + D \cdot u(t) &
\end{matrix}\right.
\]
The transfer function of $\Sigma$ is given by
\[
T_{\Sigma}(\lambda) = D+C(\lambda \cdot I - A)^{-1}B.
\]
Given a feedback operator $F$, forming a closed loop system with respect to state feedback,
one obtains the following linear system $\Sigma_F$:
\[
\Sigma_F: \left\{
\begin{matrix}
x'(t) = (A+BF)\cdot x(t) + B \cdot u(t), & x(0) = x_0\\
y(t) = (C+DF)\cdot x(t) + D \cdot u(t) &
\end{matrix}\right.
\]
The transfer function of the closed loop system is given by
\[
T_{\Sigma_F}(\lambda) = D+(C+DF)(\lambda \cdot I - (A+BF))^{-1}B.
\]
The \textit{Pole-shifting problem} asks what are the possible pole sets of the function $T_{\Sigma_F}(\lambda)$
one may obtain by changing the feedback operator $F$.

This paper discusses the \textit{Pole-shifting problem} in the context of operator vessels. 
The transfer function of an operator vessel is given by a meromorphic bundle map between two vector bundles over a compact Riemann surface given by the normalization of a plane algebraic curve. Interpolation problems for such functions are far from being trivial, and the classical approach to pole shifting using an explicit construction of the feedback operator directly from the prescribed pole datum seems difficult to achieve. 

In view of this difficulty, we propose in this paper a new approach for pole placement. We will show (Proposition \ref{prop:Factorization}) that whenever a closed loop system of an operator vessel is formed by state feedback, its transfer function factors as a composition of the transfer function of the open loop system, and the transfer function of another system, called the controller system associated to the state feedback operator. The controller system has a simpler structure, and is thus easier to construct. Our construction gives a new approach for pole placement even for classical one-dimensional linear systems,
and as a result improves over the best known result even in the classical case.

We were led to the definition of the controller system by the rigidity of the vessel conditions. Thus, this work serves as a demonstration for the principle that developing system theoretic ideas in the more complicated overdetermined 2D setting might shed new light on the classical one dimensional case. Here is the main result of this text:

\begin{thm*}
Let $\mathcal{B}$ be an operator vessel (see Definition \ref{dfn:vessel}) satisfying the assumptions (\ref{fact:polArePowers}),
(\ref{fact:maximality}), (\ref{fact:smooth}) and (\ref{fact:regular}).
Denote by $X$ the compact Riemann surface associated to it. 
Let $E_{\opn{in}}$ and $E_{\opn{out}}$ be the input and output holomorphic vector bundles over $X$ associated to $\mathcal{B}$, 
and denote by $S:E_{\opn{in}}\to E_{\opn{out}}$ the transfer function of $\mathcal{B}$. 

Given a meromorphic bundle map $T:E_{\opn{in}}\to E_{\opn{out}}$,
there exists an admissible state feedback operator  $F$, 
such that the closed loop system obtained from $\mathcal{B}$ by applying the feedback operator $F$ has a transfer function equal to $T$,
if and only if the left zero divisor of $T$ is contained in the left zero divisor of $S$, 
and $T$ is equal to $S$ at all points of $X$ which lie over the line at infinity.
\end{thm*}

This is repeated as Theorem \ref{thm:main} in the body of the paper.
It answers a question of Ball and Vinnikov (see \cite[Section 4]{BV1}).

We notice that the paper Ball--Boquet--Vinnikov \cite{BBV12} identified how vessels fit into the general framework of behavioral systems. It would be interesting to tie the study of state feedback for vessels with the general study of feedback in the behavioral framework (see e.g. \cite{R02}).

\numberwithin{equation}{subsection}

\section{Operator vessels and state feedback}\label{sec:Vessel}

\subsection{Operator vessels and their associated compact Riemann surface}

We begin by recalling the definition of an operator vessel, a notion which serves as a useful model for studying overdetermined 2D systems as in equation (\ref{eqn:ODSystem}). We refer the reader to \cite{BV1,Kr,Li,LKMV,Vi} and their references for more background about these objects.

\begin{dfn}\label{dfn:vessel}
A Liv{\v{s}}ic-Kravitsky commutative two-operator vessel (abbreviated to operator vessel, or simply a vessel) is a collection of linear operators and spaces of the form:
\[
\mathcal{B} = (A_1,A_2,\widetilde{B},C,D,\widetilde{D},\sigma_1,\sigma_2,\gamma,\sigma_{1*},\sigma_{2*},\gamma_*;\mathcal{H},\mathcal{E},\widetilde{\mathcal{E}},\mathcal{E}_*,\widetilde{\mathcal{E}_*})
\]
Here, the vector spaces $\mathcal{H},\mathcal{E},\widetilde{\mathcal{E}},\mathcal{E}_*$ and $\widetilde{\mathcal{E}}_*$ are finite dimensional vector spaces over $\mathbb{C}$,
and there are equalities $\dim \mathcal{E} = \dim \widetilde{\mathcal{E}}$, $\dim \mathcal{E}_* = \dim \widetilde{\mathcal{E}}_*$.

The operators act as follows:
 $A_1,A_2:\mathcal{H}\to \mathcal{H}$, $\widetilde{B}:\widetilde{\mathcal{E}} \to \mathcal{H}$, $\sigma_1,\sigma_2,\gamma:\mathcal{E}\to\widetilde{\mathcal{E}}$, $C:\mathcal{H}\to\mathcal{E}_*$, $D:\mathcal{E}\to\mathcal{E}_*$, $\widetilde{D}:\mathcal{E}_*\to \widetilde{\mathcal{E}}_*$ and $\sigma_{1*},\sigma_{2*},\gamma_*:\mathcal{E}_* \to \widetilde{\mathcal{E}}_*$.
 
It is assumed that the operators $D$ and $\widetilde{D}$ are invertible, 
and that the following conditions, called the vessel conditions, hold:
\begin{equation}\label{eqn:vesselCond}
\left.\begin{array}{l}
\mbox{(A1)} \quad A_1A_2 = A_2A_1\\
\mbox{(A2)} \quad A_1\tilde{B}\sigma_2-A_2\tilde{B}\sigma_1+\tilde{B}\gamma=0\\
\mbox{(A3)} \quad \sigma_{2*}CA_1-\sigma_{1*}CA_2+\gamma_*C = 0\\
\mbox{(A4)} \quad \sigma_{1*}D = \tilde{D}\sigma_1, \quad \sigma_{2*}D = \tilde{D}\sigma_2\\
\phantom{(A4)} \quad \gamma_*D = \tilde{D}\gamma + \sigma_{1*}C\tilde{B}\sigma_2-\sigma_{2*}C\tilde{B}\sigma_1.
\end{array}\right.
\end{equation}
\end{dfn}

\begin{fact}
Given an operator vessel 
\[
\mathcal{B} = (A_1,A_2,\widetilde{B},C,D,\widetilde{D},\sigma_1,\sigma_2,\gamma,\sigma_{1*},\sigma_{2*},\gamma_*;\mathcal{H},\mathcal{E},\widetilde{\mathcal{E}},\mathcal{E}_*,\widetilde{\mathcal{E}_*}),
\] 
we define two polynomials in two complex variables:
\[
\mathbf{p}_{0}(\lambda_1,\lambda_2) = \det(\lambda_1\sigma_2-\lambda_2\sigma_1+\gamma)
\]
and
\[
\mathbf{p}_{0*}(\lambda_1,\lambda_2) = 
\det(\lambda_1\sigma_{2*}-\lambda_2\sigma_{1*}+\gamma_*).
\]
\end{fact}

\begin{fact}\label{fact:polArePowers}
We will make the following assumption on the polynomials $\mathbf{p}_{0}$ and $\mathbf{p}_{0*}$: 
we assume that 
\begin{equation}\label{eqn:pin}
\mathbf{p}_{0}(\lambda_1,\lambda_2) = (\mathbf{f}_0(\lambda_1,\lambda_2))^r, 
\end{equation}
and 
\begin{equation}\label{eqn:pout}
\mathbf{p}_{0*}(\lambda_1,\lambda_2) = (\mathbf{f}_{0*}(\lambda_1,\lambda_2))^s
\end{equation}
for some irreducible polynomials $\mathbf{f}_0,\mathbf{f}_{0*} \in \mathbb{C}[\lambda_1,\lambda_2]$.
We define the following plane algebraic curves:
\[
\mathbf{C}_1 = \{(\lambda_1,\lambda_2) \in \mathbb{C}^2 | \mathbf{f}_0(\lambda_1,\lambda_2)=0\}, \quad \mathbf{C}_2 = \{(\lambda_1,\lambda_2) \in \mathbb{C}^2 | \mathbf{f}_{0*}(\lambda_1,\lambda_2)=0\}.
\]
By abuse of notation we will also denote the extensions of these affine curves to $\mathbb{P}^2\mathbb{C}$ by $\mathbf{C}_1, \mathbf{C}_2$.
Denoting by $L_{\infty}$ the line at infinity of $\mathbb{P}^2\mathbb{C}$,
we will make the assumption that for any $p \in \mathbf{C}_1$ (respectively $p \in \mathbf{C}_2$),
such that $p \in L_{\infty}$,
the intersection number of $\mathbf{C}_1$ (resp. $\mathbf{C}_2$) and $L_{\infty}$ at $p$ is equal to $1$.
\end{fact}

\begin{fact}
Given a plane algebraic curve $\mathbf{C} = \{(\lambda_1,\lambda_2) \mid f(\lambda_1,\lambda_2) = 0\}$, 
for some $f \in \mathbb{C}[x,y]$, 
and given some $p \in \mathbf{C}$,
we denote by $\mu_p(C)$ the multiplicity of $p$ on $\mathbf{C}$. 
By definition, this is the smallest integer $n$,
such that all partial derivatives of $f$ of degrees $<n$ vanish at $p$,
and at least one partial derivative of $f$ or order $n$ does not vanish at $p$.
Note that $\mathbf{C}$ is smooth at $p$ if and only if $\mu_p(C) = 1$.
\end{fact}

\begin{fact}
For any $(\lambda_1,\lambda_2) \in \mathbf{C}_1$, we consider the subspace
\[
E_{\opn{in}}(\lambda_1,\lambda_2) = \ker(\lambda_1\sigma_2-\lambda_2\sigma_1+\gamma).
\]
Similarly, for $(\lambda_1,\lambda_2) \in \mathbf{C}_2$, we define
\[
E_{\opn{out}}(\lambda_1,\lambda_2) = \ker(\lambda_1\sigma_{2*}-\lambda_2\sigma_{1*}+\gamma_*).
\]
By \cite[Proposition 10.5.1]{LKMV}, for any $p \in \mathbf{C}_1$,
and any $q \in \mathbf{C}_2$
one has inequalities
\begin{equation}\label{eqn:ineq}
\dim E_{\opn{in}}(p) \le \mu_p(\mathbf{C}_1) \cdot r, \quad \dim E_{\opn{out}}(q) \le \mu_q(\mathbf{C}_2) \cdot s.
\end{equation}
Here, $r$ and $s$ are as in (\ref{eqn:pin}) and (\ref{eqn:pout}).
Note that
$E_{\opn{in}}$ and $E_{\opn{out}}$ have the structure of torsion free sheaves over $\mathbf{C}_1, \mathbf{C}_2$.
\end{fact}

\begin{fact}\label{fact:maximality}
We will further make the maximality assumption, 
namely, that the two inequalities of (\ref{eqn:ineq}) are equalities at all points of $\mathbf{C}_1$ and $\mathbf{C}_2$. We also make a somewhat stronger assumption that $E_{\opn{in}}$ and $E_{\opn{out}}$ are fully saturated (see \cite[Section 4]{KV1}, \cite[Section 2.4.5]{KV2} or \cite[Page 340]{Vi2} for discussions about this notion).
The most important thing to note about this assumption, as explained in \cite{BV1}, is that it is satisfied if $\mathbf{C}_1$ and $\mathbf{C}_2$ are smooth algebraic curves,
and moreover either ${\mathbf p}_0$ and ${\mathbf p}_{0*}$ are irreducible polynomials, or $r=s=1$.
\end{fact}

\begin{fact}
As explained in \cite[Section 1.2]{BV1}, the assumptions (\ref{fact:polArePowers}) and (\ref{fact:maximality}) and the fact that the operator $D$ is invertible imply that there is some constant $\mu \in \mathbb{C}^{\times}$ such that $\mathbf{p}_{0*}(\lambda_1,\lambda_2) = \mu \cdot \mathbf{p}_{0}(\lambda_1,\lambda_2)$. 
Thus, under these assumptions, to any vessel $\mathcal{B}$ there is an associated plane algebraic curve 
\[
\mathbf{C} = \{(\lambda_1,\lambda_2) \in \mathbb{C}^2 \mid \mathbf{p}_{0}(\lambda_1,\lambda_2)=0\} = \{(\lambda_1,\lambda_2) \in \mathbb{C}^2 \mid \mathbf{p}_{0*}(\lambda_1,\lambda_2)=0\}.
\]
Denote by $X$ the associated compact Riemann surface obtained from the normalization of $\mathbf{C}$. 
According to \cite[Theorem 2.1]{BV2}, the torsion free sheaves $E_{\opn{in}}$ and $E_{\opn{out}}$ lift to holomorphic vector bundles over $X$. 
By abuse of notation, we will also denote them by $E_{\opn{in}}$ and $E_{\opn{out}}$.
Vector bundles that arise in such a way are called vector bundles which have a determinantal representation.
\end{fact}

\subsection{The transfer function of an operator vessel}

To discuss the transfer function associated to the vessel $\mathcal{B}$ we first recall the notion of a joint spectrum:

\begin{fact}
Let $A_1, A_2 \in M_n(\mathbb{C})$ be two square matrices. We say that $A_1,A_2$ are commuting if $A_1 \cdot A_2 = A_2 \cdot A_1$. 
In this case, their joint spectrum $\opn{Spec}(A_1,A_2)$ is defined to be the set of all pairs $(\lambda_1,\lambda_2) \in \mathbb{C}^2$, 
such that $A_1 \cdot v = \lambda_1 \cdot v$ and $A_2 \cdot v = \lambda_2 \cdot v$ for some non-zero vector $v\in \mathbb{C}^n$. 
The following easy fact from linear algebra characterizes the joint spectrum: 
for any $(\lambda_1,\lambda_2)\in \mathbb{C}^2$, there are $\xi_1,\xi_2 \in \mathbb{C}$ such that $\xi_1(\lambda_1I-A_1)+\xi_2(\lambda_2I-A_2)$ is invertible if and only if $(\lambda_1,\lambda_2) \notin \opn{Spec}(A_1,A_2)$. 
\end{fact}

\begin{fact}
The transfer function of $\mathcal{B}$ is defined as follows: given 
$(\lambda_1,\lambda_2) \in \mathbf{C}$, 
such that 
\[
(\lambda_1,\lambda_2) \notin \opn{Spec}(A_1,A_2),
\]
let $\xi_1,\xi_2 \in \mathbb{C}$ be such that 
\[
\xi_1(\lambda_1I-A_1)+\xi_2(\lambda_2I-A_2)
\]
is invertible. 
For any $v \in E_{\opn{in}}(\lambda_1,\lambda_2)$, we define:
\[
S_{\mathcal{B}}(\lambda_1,\lambda_2)v = (D+C(\xi_1(\lambda_1I-A_1)+\xi_2(\lambda_2I-A_2))^{-1}\widetilde{B}(\xi_1\sigma_1+\xi_2\sigma_2))v.
\]
It was shown in \cite{BV1} that this is independent of the choices of $\xi_1,\xi_2$,
and that 
\[
S_{\mathcal{B}}(\lambda_1,\lambda_2)v \in E_{\opn{out}}(\lambda_1,\lambda_2).
\] 
This means that $S_{\mathcal{B}}$ is a bundle map, defined outside the finite set 
$\opn{Spec}(A_1,A_2)$. 
This map may be lifted to a meromorphic bundle map $E_{\opn{in}} \to E_{\opn{out}}$ over $X$ which is holomorphic outside points which lie over $\opn{Spec}(A_1,A_2)$ (and may or may not have poles at the points above the joint spectrum). 
Another useful property of $S_{\mathcal{B}}$ which follows from this definition is that $S_{\mathcal{B}}$ is equal to $D$ when restricted to the points of $X$ which lie over $L_{\infty}$.
\end{fact}

\begin{fact}\label{fact:smooth}
In view of the above discussion, we will make the following additional assumption on the vessel $\mathcal{B}$: every point $\lambda$ in the joint spectrum $\opn{Spec}(A_1,A_2)$ is a smooth point of $\mathbf{C}$. This ensures that the singularities of $S_{\mathcal{B}}$ lie all over the smooth points of $\mathbf{C}$, and there are no poles at singular points.
\end{fact}

We next discuss a class of functions that the transfer function belongs to. 

\begin{fact}
Let $X$ be a compact Riemann surface, and let $\pi_E:E \to X ,\pi_F:F \to X$ be two holomorphic vector bundles over $X$. 
In particular, $E$ and $F$ are complex manifolds, so it makes sense to talk about holomorphic and meromorphic functions between them. A map $T:E\to F$ is called a \textbf{meromorphic bundle map} if it is a meromorphic map which is also a bundle map, that is: $\pi_F \circ T = \pi_E$, and $T$ is linear over each fiber in which it is defined. 
\end{fact}

\begin{fact}
The transfer function of an operator vessel is an example of a meromorphic bundle map. 
For vector bundles which have determinantal representations, the converse is also true: every meromorphic bundle map (which is regular at the points at infinity) between such bundles is the transfer function of some operator vessel (see \cite{BV2,BV,BV3} for a proof of this fact). 
\end{fact}

To discuss zero and pole data of meromorphic bundle maps, we follow the local case, as in \cite{BGR}. Given $p \in \mathbb{C}$, we denote by $\mathcal{O}_p$ the ring of germs of holomorphic functions at $p$, and by $\mathcal{O}^{\times}_p$ its subset consisting of germs $\phi$ such that $\phi(p) \ne 0$.

\begin{fact}
Let $A(z)$ be an order $m$ square matrix of meromorphic functions near some point $z_0$,
such that $\det A(z)$ is not identically zero. 
This implies that $A^{-1}$ is also a meromorphic matrix function near $z_0$.
Denoting by $\mathcal{O}^{\times,m}$ the set $\left\{\phi \in {\mathcal O}_{z_0}^{1 \times m} \colon \phi(z_0) \neq 0\right\}$,
given $\phi \in \mathcal{O}^{\times,m}_{z_0}$,
we say that $A$ has a left zero at the point $z_0$ in direction $\phi$ of order $n$, if 
$\phi(z) A(z) = z^n \psi(z)$ for some $\psi \in \mathcal{O}^{\times,m}_{z_0}$. 
We say that $A(z)$ has a left pole at $z_0$ in direction $\phi(z)$ of order $n$ if  $A^{-1}(z)$ has a left zero at $z_0$ in direction $\phi$ of order $n$.
\end{fact}

These definitions are generalized to the global case of meromorphic bundle maps by replacing holomorphic germs by germs of holomorphic sections. We define the divisor datum of a meromorphic bundle map as follows:

\begin{dfn}\label{dfn:leftpole}
Let $X$ be a compact Riemann surface,
and let $T: E \to F$ be a meromorphic bundle map between two holomorphic vector bundles over $X$.
\begin{enumerate}
\item The left zero set of $T$ is the set 
\[
\opn{LZ}(T) = \{ (\phi,n,z_0) \mid \mbox{$T$ has a left zero at $z_0$ of order $\ge n$ at direction $\phi$}\}.
\]
\item The left pole set of $T$ is the set 
\[
\opn{LP}(T) = \{ (\phi,n,z_0) \mid \mbox{$T$ has a left pole at $z_0$ of order $\ge n$ at direction $\phi$}\}.
\]
\end{enumerate}
\end{dfn}

\begin{fact}\label{fact:whereZerosLive}
Note that by definition, a left zero of $T:E \to F$ is a triple $(\phi,n,z_0)$,
where $\phi$ is a germ of an holomorphic section of the bundle $F^*$, the dual of the bundle $F$. 
Similarly, a left pole of $T$ is a triple $(\phi,n,z_0)$,
where $\phi$ is a germ of an holomorphic section of the bundle $E^*$.
\end{fact}

\subsection{Controllability and Observability of operator vessels}

\begin{fact}
A one dimensional linear system $\Sigma=(A,B,C,D;\mathcal{H},\mathcal{E},\mathcal{E}_*)$ is
called controllable if the pair $(A,B)$ is controllable.
Explicitly, this means that 
\[
\sum_{n=0}^{\infty} \opn{Im} A^nB = \mathcal{H}.
\]
Similarly, $\Sigma$ is called observable if the pair$(C,A)$ is observable.
That is, 
\[
\bigcap_{n=0}^{\infty} \ker(CA^n) = \{0\}.
\]
These linear algebra definitions are equivalent to the usual system-theoretic definitions of these terms. 
\end{fact}

Similarly, for operator vessels, we define:

\begin{dfn}\label{def:COM}
Let $\mathcal{B} = (A_1,A_2,\tilde{B},C,D,\tilde{D},\sigma_1,\sigma_2,\gamma,\sigma_{1*},\sigma_{2*},\gamma_*;\mathcal{H},\mathcal{E},\tilde{\mathcal{E}},\mathcal{E}_*,\tilde{\mathcal{E}_*})$ be an operator vessel.
\begin{enumerate}
\item We say that $\mathcal{B}$ is controllable if 
\[
\sum_{n_1=0}^{\infty} \sum_{n_2=0}^{\infty} \opn{Im} {A_1}^{n_1}{A_2}^{n_2}\widetilde{B} = \mathcal{H}.
\]
\item We say that $\mathcal{B}$ is observable if 
\[
\bigcap_{n_1=0}^{\infty}\bigcap_{n_2=0}^{\infty} \ker (C{A_1}^{n_1}{A_2}^{n_2}) = \{0\}.
\]
\item The operator vessel $\mathcal{B}$ is called minimal if it is both controllable and observable.
\end{enumerate}
\end{dfn}

\begin{fact}
In \cite[Proposition 1.11]{BV1}, it was shown that as in the one dimensional case, 
one may give system-theoretic definitions to these terms, imitating the usual ones in terms of the controllable subspace and unobservable subspace, and that they are equivalent to Definition \ref{def:COM}. As we will not need these in this paper, we omit recalling them.
\end{fact}

\begin{fact}\label{fact:CH}
If a vessel $\mathcal{B}$ is either controllable or observable,
the generalized Generalized Cayley-Hamilton Theorem (\cite[Section 8.2]{LKMV}) implies that $\opn{Spec}(A_1,A_2) \subseteq \mathbf{C}$.
\end{fact}

\begin{fact}
Given a vessel 
\[
\mathcal{B} = (A_1,A_2,\widetilde{B},C,D,\widetilde{D},\sigma_1,\sigma_2,\gamma,\sigma_{1*},\sigma_{2*},\gamma_*;\mathcal{H},\mathcal{E},\widetilde{\mathcal{E}},\mathcal{E}_*,\widetilde{\mathcal{E}_*})
\]
we say that a direction $(\xi_1,\xi_2) \in \mathbb{P}^2\mathbb{C}$ is a regular direction for $\mathcal{B}$ if the operator $\sigma_{\xi}:=\xi_1\sigma_1 + \xi_2\sigma_2$ is invertible. By the vessel condition (A4), this implies that the operator $\sigma_{*\xi} = \xi_1\sigma_{1*} + \xi_2\sigma_{2*}$ is also invertible. 
\end{fact}

\begin{fact}\label{fact:regular}
We will make the following assumption: all vessels in this paper have regular directions.
Equivalently, the function $\det(\sigma_{\xi})$ (equivalently, $\det(\sigma_{*\xi})$) is not identically zero.
\end{fact}

\begin{fact}
Given a direction $\xi = (\xi_1,\xi_2)$, we will shorten notation and set
\[
A_{\xi} = \xi_1A_1 + \xi_2A_2, \quad B_{\xi} = \widetilde{B}(\xi_1\sigma_1+\xi_2\sigma_2).
\]

Using these notations, we define:
\begin{equation}\label{eqn:RestrictedFunction}
S_{\xi}(\lambda) = D+C(\lambda I-A_{\xi})^{-1}B_{\xi}
\end{equation}
The function $S_{\xi}(\lambda)$ is rational matrix function, called the restricted transfer function of $\mathcal{B}$ at direction $(\xi_1,\xi_2)$.
\end{fact}

\begin{fact}
Given a rational matrix function $S(\lambda)$, 
a system theoretic realization of $S$ is a presentation:
\[
S(\lambda) = D+C(\lambda I-A)^{-1}B.
\]
Such a presentation is called minimal if the square matrix $A$ has minimal size along all possible realizations of $S$. By \cite[Theorem 4.1.4]{BGR}, this happens if and only if the pair $(C,A)$ is observable, and the pair $(A,B)$ is controllable.
\end{fact}

\begin{prop}\label{prop:MinimalRealization}
Let $\mathcal{B}$ be a minimal vessel, 
and let $(\xi_1,\xi_2) \in \mathbb{P}^2\mathbb{C}$ be a regular direction for $\mathcal{B}$. 
Then the realization (\ref{eqn:RestrictedFunction}) of the restricted transfer function $S_{\xi}$ is minimal.
\end{prop}

This follows from the next two lemmas:

\begin{lem}\label{lem:obsReg}
Suppose $\OpenLoop$ is an observable vessel, and suppose that $\xi$ is a regular direction for $\OpenLoop$. Then
\[
\bigcap_{n=0}^{\infty} \ker CA_{\xi}^n = \{0\}
\]
\end{lem}
\begin{proof}
Since $\OpenLoop$ is observable, we have that
\[
\bigcap_{n_1=0}^{\infty} \bigcap_{n_2=0}^{\infty} \ker (CA_1^{n_1}A_2^{n_2}) = \{0\}
\]
Since $\xi$ is regular, it follows that $\sigma_{*\xi}$ is invertible. 
By the vessel condition (A3) we have that 
\[
\sigma_{2*}CA_1-\sigma_{1*}CA_2+\gamma_*C = 0
\]
Multiplying both sides of this equation by $\xi_1\xi_2$ and rearranging we get
\[
CA_1 = \frac{1}{\xi_1}\sigma_{*\xi}^{-1}(\xi_1\sigma_{1*}CA_{\xi}-\xi_1\xi_2\gamma_{*}C)
\]
and similarly
\[
CA_2 = \frac{1}{\xi_2}\sigma_{*\xi}^{-1}(\xi_2\sigma_{2*}CA_{\xi}+\xi_1\xi_2\gamma_*C)
\]
so that both $CA_1$ and $CA_2$ are of the form $E_1C+E_2CA_{\xi}$ for some matrices $E_1,E_2$. We now claim that for all $n_1\ge 0$, $n_2\ge 0$ one may write $CA_1^{n_1}A_2^{n_2} = \sum_{k=0}^{n_1+n_2} E_k C A_{\xi}^k$ for some matrices $E_0,\dots,E_{n_1+n_2}$. We prove this by induction. By symmetry and since $A_1$ and $A_2$ commute, it is enough to show that if it is true for $CA_1^{n_1}A_2^{n_2}$ then it is true for $CA_1^{n_1}A_2^{n_2+1}$.
Let $CA_2 = M_0C + M_1CA_{\xi}$.
Write
\[
CA_1^{n_1}A_2^{n_2} = \sum_{k=0}^{n_1+n_2} E_k C A_{\xi}^k
\]
and multiply this by $A_2$. Then
\[
CA_1^{n_1}A_2^{n_2+1} = \sum_{k=0}^{n_1+n_2} E_k C A_{\xi}^k A_2
\]
However, $A_2$ and $A_{\xi}$ commute, so we may write each term as:
\[
E_k C A_{\xi}^k A_2 = E_k CA_2 A_{\xi}^k = E_k (M_0C+M_1CA_{\xi}) A_{\xi}^k = 
E_kM_0CA_{\xi}^k + E_kM_1CA_{\xi}^{k+1}
\]
so the entire sum has the required form. Thus, we obtain that for all $n_1\ge 0$, $n_2\ge 0$ we have that
\[
\bigcap_{k=0}^{n_1+n_2} \ker (CA_{\xi}^k) \subseteq \ker(CA_1^{n_1}A_2^{n_2})
\]
so the result follows.\qed
\end{proof}
Dually, and using the vessel condition (A2) one has that
\begin{lem}
Suppose $\OpenLoop$ is a controllable vessel, and suppose that $\xi$ is a regular direction for $\OpenLoop$. Then
\[
\sum_{n=0}^{\infty} \opn{Im} A_{\xi}^n B_{\xi} = \mathcal{H}
\]
\end{lem}

\subsection{State feedback for operator vessels}

Following \cite[Example 1.20]{BV1},
we now introduce state feedback for operator vessels.
Because of the centrality of this construction to this paper,
we verify the following in details, even though it is a bit tedious.

\begin{prop}
Let 
\[
\mathcal{B} = (A_1,A_2,\widetilde{B},C,D,\widetilde{D},\sigma_1,\sigma_2,\gamma,\sigma_{1*},\sigma_{2*},\gamma_*;\mathcal{H},\mathcal{E},\widetilde{\mathcal{E}},\mathcal{E}_*,\widetilde{\mathcal{E}_*})
\]
be a vessel, and let $F:\mathcal{H}\to \mathcal{E}$ be a linear operator.
Suppose that $F$ satisfies the following two conditions
\begin{equation}\label{eqn:feedback1}
\sigma_2 F A_1 - \sigma_1 F A_2 + \gamma F =0,
\end{equation}
\begin{equation}\label{eqn:feedback2}
\sigma_1 F \widetilde{B}\sigma_2 - \sigma_2 F \widetilde{B} \sigma_1 = 0
\end{equation}
then the collection $\ClosedLoop =$
\[
(A_1+\widetilde{B}\sigma_1F,A_2+\widetilde{B}\sigma_2F,\widetilde{B},C+DF,D,\widetilde{D},\sigma_1,\sigma_2,\gamma,\sigma_{1*},\sigma_{2*},\gamma_*;\mathcal{H},\mathcal{E},\widetilde{\mathcal{E}},\mathcal{E}_*,\widetilde{\mathcal{E}_*})
\]
is an operator vessel. 
We say that $F$ is an admissible state feedback operator, 
and that $\ClosedLoop$ is the closed loop system formed by applying the state feedback operator $F$.
\end{prop}
\begin{proof}
Suppose $F$ satisfies equations (\ref{eqn:feedback1}) and (\ref{eqn:feedback2}). We have to verify the four vessel conditions.\\
Condition (A1):
\begin{equation}
\left.\begin{array}{l}
(A_1+\widetilde{B}\sigma_1F)(A_2+\widetilde{B}\sigma_2F) - (A_2+\widetilde{B}\sigma_2F)(A_1+\widetilde{B}\sigma_1F) =
\\
= (A_1A_2+A_1\widetilde{B}\sigma_2F+\widetilde{B}\sigma_1FA_2+\widetilde{B}\sigma_1F\widetilde{B}\sigma_2F)\\
- (A_2A_1+A_2\widetilde{B}\sigma_1F+\widetilde{B}\sigma_2F\widetilde{B}\sigma_1F+\widetilde{B}\sigma_2FA_1) =\\
=
(A_1A_2-A_2A_1) + (\widetilde{B}\sigma_1F\widetilde{B}\sigma_2F-\widetilde{B}\sigma_2F\widetilde{B}\sigma_1F) +\\
+(A_1\widetilde{B}\sigma_2F+\widetilde{B}\sigma_1FA_2-A_2\widetilde{B}\sigma_1F-\widetilde{B}\sigma_2FA_1)
\end{array}\right.
\end{equation}
The first term vanishes because of vessel condition (A1) satisfied by $\mathcal{B}$. The second term vanishes because of (\ref{eqn:feedback2}). For the third term:
\begin{eqnarray}
A_1\widetilde{B}\sigma_2F+\widetilde{B}\sigma_1FA_2-A_2\widetilde{B}\sigma_1F-\widetilde{B}\sigma_2FA_1 =\nonumber\\ A_1\widetilde{B}\sigma_2F-A_2\widetilde{B}\sigma_1F+\widetilde{B}(\sigma_1FA_2-\sigma_2FA_1)\nonumber
\end{eqnarray}
From (\ref{eqn:feedback1}), we have $\sigma_1FA_2-\sigma_2FA_1 = \gamma F$, so the last term is equal to
\begin{equation}
= A_1\widetilde{B}\sigma_2F-A_2\widetilde{B}\sigma_1F+\widetilde{B}\gamma F = (A_1\widetilde{B}\sigma_2-A_2\widetilde{B}\sigma_1+\widetilde{B}\gamma)F = 0
\end{equation}
where the last equality follows from the vessel condition (A2) for the vessel $\mathcal{B}$. This establishes (A1).\\
Condition (A2):
\begin{equation}
\left.\begin{array}{l}
(A_2+\widetilde{B}\sigma_2F)\widetilde{B}\sigma_1-(A_1+\widetilde{B}\sigma_1F)\widetilde{B}\sigma_2-\widetilde{B}\gamma =\\
= (A_2\widetilde{B}\sigma_1-A_1\widetilde{B}\sigma_2-\widetilde{B}\gamma) + (\widetilde{B}\sigma_2F\widetilde{B}\sigma_1-\widetilde{B}\sigma_1F\widetilde{B}\sigma_2) = 0
\end{array}\right.
\end{equation}
where the first term vanishes because of the vessel condition (A2) of $\mathcal{B}$, and the vanishing of the second term follows from (\ref{eqn:feedback2}).\\
Condition (A3):
\begin{equation}\label{VerifyA3forClsoedLoop}
\left.\begin{array}{l}
\sigma_{2*}(C+DF)(A_1+\widetilde{B}\sigma_1F)
-\sigma_{1*}(C+DF)(A_2+\widetilde{B}\sigma_2F) + \gamma_*(C+DF)=\\
=\sigma_{2*}(CA_1+C\widetilde{B}\sigma_1F+DFA_1+DF\widetilde{B}\sigma_1F)\\
-\sigma_{1*}(CA_2+C\widetilde{B}\sigma_2F+DFA_2+DF\widetilde{B}\sigma_2F) + \gamma_*C+\gamma_*DF =\\
(\sigma_{2*}CA_1-\sigma_{1*}CA_2+\gamma_*C) + (\sigma_{2*}C\widetilde{B}\sigma_1F-\sigma_{1*}C\widetilde{B}\sigma_2F+\gamma_*DF) +\\
+ (\sigma_{2*}DFA_1-\sigma_{1*}DFA_2+\sigma_{2*}DF\widetilde{B}\sigma_1F-\sigma_{1*}DF\widetilde{B}\sigma_2F)
\end{array}\right.
\end{equation}
The first term vanishes because of condition (A3) for the vessel $\mathcal{B}$. For the third term, using the equation $\sigma_{i*}D = \widetilde{D}\sigma_i$ (condition (A4) for $\mathcal{B}$), we have
\begin{equation}
\left.\begin{array}{l}
\sigma_{2*}DFA_1-\sigma_{1*}DFA_2+\sigma_{2*}DF\widetilde{B}\sigma_1F-\sigma_{1*}DF\widetilde{B}\sigma_2F =\\
= \widetilde{D}\sigma_2FA_1-\widetilde{D}\sigma_1FA_2 +\widetilde{D}\sigma_2F\widetilde{B}\sigma_1F -\widetilde{D}\sigma_1F\widetilde{B}\sigma_2F = \\
=
\widetilde{D} ((\sigma_2FA_1-\sigma_1FA_2) + (\sigma_2F\widetilde{B}\sigma_1-\sigma_1F\widetilde{B}\sigma_2)F) = -\widetilde{D}\gamma F
\end{array}\right.
\end{equation}
where the last equation follows from (\ref{eqn:feedback1}) and (\ref{eqn:feedback2}). Hence, equation (\ref{VerifyA3forClsoedLoop}) becomes
\begin{equation}\label{VerifyA3forClsoedLoopCont}
(\sigma_{2*}C\widetilde{B}\sigma_1F-\sigma_{1*}C\widetilde{B}\sigma_2F+\gamma_*DF) -\widetilde{D}\gamma F
\end{equation}
using the relation $\widetilde{D}\gamma = \gamma_*D-\sigma_{1*}C\widetilde{B}\sigma_2+\sigma_{2*}C\widetilde{B}\sigma_1$ (condition (A4) for the vessel $\mathcal{B}$), the equation (\ref{VerifyA3forClsoedLoopCont}) becomes
\begin{equation}
(\sigma_{2*}C\widetilde{B}\sigma_1F-\sigma_{1*}C\widetilde{B}\sigma_2F+\gamma_*DF) -(\gamma_*D-\sigma_{1*}C\widetilde{B}\sigma_2+\sigma_{2*}C\widetilde{B}\sigma_1)F = 0
\end{equation}
This establishes (A3).\\
Condition (A4):\\
The equations $\sigma_{i*}D = \widetilde{D}\sigma_i$ are satisfied because of the vessel condition (A4) of $\mathcal{B}$. We now verify the last equation of (A4):
\begin{equation}
\left.\begin{array}{l}
\gamma_*D-\widetilde{D}\gamma-\sigma_{1*}(C+DF)\widetilde{B}\sigma_2+\sigma_{2*}(C+DF)\widetilde{B}\sigma_1 =\\
=(\gamma_*D-\widetilde{D}\gamma-\sigma_{1*}C\widetilde{B}\sigma_2+\sigma_{2*}C\widetilde{B}\sigma_1) + (\sigma_{2*}DF\widetilde{B}\sigma_1-\sigma_{1*}DF\widetilde{B}\sigma_2).
\end{array}\right.
\end{equation}
The vanishing of the first term follows from condition (A4) of $\mathcal{B}$. For the second term, using the relation $\sigma_{i*}D = \widetilde{D}\sigma_i$ we obtain
\begin{equation}
\sigma_{2*}DF\widetilde{B}\sigma_1-\sigma_{1*}DF\widetilde{B}\sigma_2 = \widetilde{D}\sigma_2F\widetilde{B}\sigma_1-\widetilde{D}\sigma_1F\widetilde{B}\sigma_2 = 
\widetilde{D}(\sigma_2F\widetilde{B}\sigma_1-\sigma_1F\widetilde{B}\sigma_2) = 0
\end{equation}
where the last equality follows from (\ref{eqn:feedback2}).\\
Hence, $\ClosedLoop$ satisfies (A1)-(A4), so it is indeed a vessel.
\end{proof}

We may now state the question this paper answers: Let $\mathcal{B}$ be a minimal vessel. Which transfer functions may be obtained as transfer functions of closed loop systems obtained from $\mathcal{B}$ by state feedback? The next section will be dedicated to answer this question.

We finish this section with the following construction: state space similarity for vessels. We omit the proof which is a trivial verification, similar to the above, but easier.

\begin{prop}\label{prop:StateSpaceSim}
Let $\mathcal{B} = (A_1,A_2,\widetilde{B},C,D,\widetilde{D},\sigma_1,\sigma_2,\gamma,\sigma_{1*},\sigma_{2*},\gamma_*;\mathcal{H},\mathcal{E},\widetilde{\mathcal{E}},\mathcal{E}_*,\widetilde{\mathcal{E}_*})$ be a vessel. Given an isomorphism $N:\mathcal{H}'\rightarrow \mathcal{H}$, 
the collection $N^{-1}\mathcal{B}N$ given by
\[ 
(N^{-1}A_1N,N^{-1}A_2N,N^{-1}\tilde{B},CN,D,\tilde{D},\sigma_1,\sigma_2,\gamma,\sigma_{1*},\sigma_{2*},\gamma_*;\mathcal{H}',\mathcal{E},\tilde{\mathcal{E}},\mathcal{E}_*,\tilde{\mathcal{E}_*})
\]
is an operator vessel, and $S_{N^{-1}\mathcal{B}N} = S_{\mathcal{B}}$.
\end{prop}

\section{The pole placement theorem}

\subsection{The controller vessel}

The following is the main tool used in the proof of the main result of this paper.

\begin{prop}\label{prop:Controller}
Consider an operator vessel
\[
\mathcal{B} = (A_1,A_2,\widetilde{B},C,D,\widetilde{D},\sigma_1,\sigma_2,\gamma,\sigma_{1*},\sigma_{2*},\gamma_*;\mathcal{H},\mathcal{E},\widetilde{\mathcal{E}},\mathcal{E}_*,\widetilde{\mathcal{E}_*})
\]
and a linear operator $F:\mathcal{H}\rightarrow \mathcal{E}$. 
Then $F$ is an admissible state feedback operator if and only if the collection
\begin{equation}
\Controller = (A_1,A_2,\widetilde{B},-F,I,I,\sigma_1,\sigma_2,\gamma,\sigma_1,\sigma_2,\gamma;\mathcal{H},\mathcal{E},\widetilde{\mathcal{E}},\mathcal{E},\widetilde{\mathcal{E}})
\end{equation}
is an operator vessel. 
The vessel $\Controller$ is called the controller vessel of the state feedback operator $F$.
\end{prop}
\begin{proof}
Conditions (A1) and (A2) are exactly the two conditions of equations (\ref{eqn:feedback1}) and (\ref{eqn:feedback2}). Assuming $F$ is an admissible state feedback operator, we leave
the (easy) verification of the vessel conditions (A3)-(A4) to the reader.
\end{proof}

The next proposition explains the importance of the controller vessel for the pole placement problem.

\begin{prop}\label{prop:Factorization}
Let $\mathcal{B}$ be a vessel, and let $F$ be an admissible state feedback. 
Then the transfer function $S_{\mathcal{B}}$ associated to $\mathcal{B}$ factors as follows:
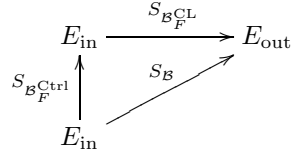
\[\xymatrixcolsep{4pc}
\xymatrix{
E_{\opn{in}} \ar[r]^{S_{\ClosedLoop}} & E_{\opn{out}}\\
E_{\opn{in}} \ar[u]^{S_{\Controller}}\ar[ru]^{S_{\mathcal{B}}}
}
\]
\end{prop}
\begin{proof}
Let $\mathbf{C}$ be the plane curve associated to $\OpenLoop$, and let $\lambda = (\lambda_1,\lambda_2) \in \mathbf{C}$. 
Suppose further that $\lambda$ is not a pole of either of the three transfer functions above. 
Let us set $S(\lambda) = S_{\OpenLoop}(\lambda)$, $T(\lambda) = S_{\ClosedLoop}(\lambda)$,
and $R(\lambda) = S_{\Controller}(\lambda)$. 
By continuity, it is enough to show that $T(\lambda) = S(\lambda)R^{-1}(\lambda)$.

We have that
\[
S(\lambda) = D+C(\xi_1(\lambda_1I-A_1)+\xi_2(\lambda_2I-A_2))^{-1}\tilde{B}(\xi_1\sigma_1+\xi_2\sigma_2)
\]
and
\[
T(\lambda) = D+(C+DF)(\xi_1(\lambda_1I-A_1-\tilde{B}\sigma_1F)+\xi_2(\lambda_2I-A_2-\tilde{B}\sigma_2F))^{-1}\tilde{B}(\xi_1\sigma_1+\xi_2\sigma_2)
\]
By \cite[Theorem 1.16]{BV1}, we have that:
\[
R^{-1}(\lambda) = I+F(\xi_1(\lambda_1I-A_1-\tilde{B}\sigma_1F)+\xi_2(\lambda_2I-A_2-\tilde{B}\sigma_2F))^{-1}\tilde{B}(\xi_1\sigma_1+\xi_2\sigma_2)
\]
where we have chosen a direction $(\xi_1,\xi_2)$ so that all of these will be well defined. This is possible because for every $\lambda$ (which is not a pole) there are only finitely many choices of directions $(\xi_1,\xi_2)$ in which the the above expressions are not well defined.
To shorten notation, let us set $V=\tilde{B}(\xi_1\sigma_1+\xi_2\sigma_2)$, 
$N = (\xi_1(\lambda_1I-A_1)+\xi_2(\lambda_2I-A_2))$, and $M = (\xi_1(\lambda_1I-A_1-\tilde{B}\sigma_1F)+\xi_2(\lambda_2I-A_2-\tilde{B}\sigma_2F))$.
Under these notations, we have
$S=D+CN^{-1}V$, $T = D+(C+DF)M^{-1}V$ and $R^{-1} = I+FM^{-1}V$.
Hence,
\[
SR^{-1} = (D+CN^{-1}V)(I+FM^{-1}V) = D + DFM^{-1}V+CN^{-1}V+CN^{-1}VFM^{-1}V
\]
Thus, we must show that 
\[
CN^{-1}V+CN^{-1}VFM^{-1}V = CM^{-1}V,
\] 
or that
\[
(CN^{-1}+CN^{-1}VFM^{-1})V = CM^{-1}V.
\] 
For this it is enough to show that
$CN^{-1}+CN^{-1}VFM^{-1} = CM^{-1}$ which is equivalent to 
$CN^{-1}M+CN^{-1}VF = C$. To show this it is enough to show that
$N^{-1}M+N^{-1}VF = I$, which is equivalent to $M+VF = N$. Since this is true, the result follows.
\end{proof}
\subsection{Left pole datum of operator vessels}
Recall that the notion of a left pole set was defined in Definition \ref{dfn:leftpole}(2).

\begin{fact}\label{fact:classical}
It is a well known fact in classical system theory that given a minimal realization 
\[
S(\lambda) = D+C(\lambda I-A)^{-1}B.
\]
of a rational matrix function $S$, 
one can read the left pole data of $S$ from the pair $(A,B)$. 
See \cite[Theorem 4.2.1(iii)]{BGR} for a precise statement of this idea.
\end{fact}

The purpose of this section is to generalize this fact to the setting of operator vessels,
by showing that the triple $(A_1,A_2,\tilde{B})$ contains the data of the left pole set of the transfer function of a minimal operator vessel.
\begin{lem}\label{lem:TtoVess}
Let 
\[
\mathcal{B} = (A_1,A_2,\tilde{B},C,D,\tilde{D},\sigma_1,\sigma_2,\gamma,\sigma_{1*},\sigma_{2*},\gamma_*;\mathcal{H},\mathcal{E},\tilde{\mathcal{E}},\mathcal{E}_*,\tilde{\mathcal{E}_*})
\]
be a controllable operator vessel,
and assume that 
let $\xi = (\xi_1,\xi_2) \in \mathbb{C}^2$ is a regular direction for $\mathcal{B}$.
Let 
\[
\mathcal{V} = (A_1',A_2',\tilde{B}',C',D',\tilde{D}',\sigma_1,\sigma_2,\gamma,\sigma_{1*}',\sigma_{2*}',\gamma_*';\mathcal{H}',\mathcal{E},\tilde{\mathcal{E}},\mathcal{E}_*',\tilde{\mathcal{E}_*}')
\]
be another operator vessel, and let $T:\mathcal{H} \to \mathcal{H}'$ be a linear operator which satisfies:
\begin{equation}\label{eqn:TA1}
T A_{\xi} = A'_{\xi}T
\end{equation}
and 
\begin{equation}\label{eqn:TB}
T\tilde{B} = \tilde{B}'.
\end{equation}
Then it also satisfies
\[
T A_1 = A_1' T, \quad T A_2 = A_2' T.
\]
\end{lem}
\begin{proof}
Let us assume without loss of generality that $\xi = (1,0)$
(the general case is almost identical, only slightly more technically involved).
Then (\ref{eqn:TA1}) implies that $TA_1 = A_1'$.
Consider the vessel condition (A2) for the vessel $\mathcal{B}$:
\[
A_1 \tilde{B}\sigma_2 - A_2 \tilde{B} \sigma_1 + \tilde{B} \gamma = 0.
\]
Multiplying this relation with $T$ from the left, and using (\ref{eqn:TA1}) and (\ref{eqn:TB})
we obtain:
\[
A_1'\tilde{B}'\sigma_2  + \tilde{B}' \gamma  =  TA_2\tilde{B}\sigma_1
\]
On the other hand, the vessel condition (A2) for the vessel $\mathcal{V}$ implies that:
\[
A_1' \tilde{B}'\sigma_2 + \tilde{B}' \gamma = A_2' \tilde{B}' \sigma_1.
\]
These two equations imply that
\[
TA_2\tilde{B}\sigma_1 = A_2' \tilde{B}' \sigma_1
\]
and using (\ref{eqn:TB}) again, and by canceling $\sigma_1$, 
which by our regularity assumption on the direction $(1,0)$ is invertible,
we deduce that
\begin{equation}\label{eqn:TA20}
TA_2\tilde{B} = A_2'T\tilde{B}.
\end{equation}
We now prove by induction that for all $n \ge 0$ there is an equality
\begin{equation}\label{eqn:TAn}
TA_2A_1^n\tilde{B} = A_2'TA_1^n\tilde{B}.
\end{equation}
The case $n = 0$ is (\ref{eqn:TA20}) established above.
Assume by induction that (\ref{eqn:TAn} holds for a given $n$.
Since $A_1$ and $A_2$ commute, we obtain:
\[
TA_2A_1^{n+1}\tilde{B} = TA_1A_2A_1^n\tilde{B} = A_1'TA_2A_1^n\tilde{B} = A_1'A_2'TA_1^n\tilde{B}
\]
where the last equality follows from the induction hypothesis.
Since $A_1'$ and $A_2'$ also commute, we have that:
\[
A_1'A_2'TA_1^n\tilde{B} = A_2'A_1'TA_1^n\tilde{B} = A_2'TA_1^{n+1}\tilde{B},
\]
as claimed.
The result now follows from (\ref{eqn:TAn}) and the controllability assumption.
\end{proof}

\begin{lem}\label{lem:newVessel}
In the situation of Lemma \ref{lem:TtoVess},
the collection
\[
\mathcal{V}' = (A_1,A_2,\tilde{B},C'T,D',\tilde{D}',\sigma_1,\sigma_2,\gamma,\sigma_{1*}',\sigma_{2*}',\gamma_*';\mathcal{H},\mathcal{E},\tilde{\mathcal{E}},\mathcal{E}_*',\tilde{\mathcal{E}_*}')
\]
is an operator vessel which satisfies
\[
S_{\mathcal{V}'} = S_{\mathcal{V}}.
\]
\end{lem}
\begin{proof}
Let us first verify that $\mathcal{V}'$ is a vessel. 
The vessel conditions (A1) and (A2) are satisfied because $\mathcal{B}$ is a vessel.
Since $\mathcal{V}$ is a vessel, its (A3) condition says that
\[
\sigma_{2*}'C'A_1' - \sigma_{1*}'C'A_2' + \gamma_*'C' = 0
\]
Multiplying this from the right by $T$, and using the fact that $A_1'T = TA_1$ and $A_2'T = TA_2$ 
implies that
\[
\sigma_{2*}'C'TA_1 - \sigma_{1*}'C'TA_2 + \gamma_*'C'T = 0
\]
so that $\mathcal{V}'$ satisfies (A3).
The first two equations of condition (A4) are satisfied for $\mathcal{V}'$ because they are satisfied for $\mathcal{V}$.
As for the last equation of (A4),
we must check that
\[
\gamma_*'D' = \tilde{D}'\gamma + \sigma_{1*}'C'T\tilde{B}\sigma_2 - \sigma_{2*}'C'T\widetilde{B}\sigma_1.
\]
Since $T\tilde{B} = \tilde{B}'$, this is equivalent to 
\[
\gamma_*'D' = \tilde{D}'\gamma + \sigma_{1*}'C'\tilde{B}'\sigma_2 - \sigma_{2*}'C'\widetilde{B}'\sigma_1,
\]
and this follows from the (A4) condition of $\mathcal{V}$.
This proves that $\mathcal{V}'$ is an operator vessel.
It remains to show that $S_{\mathcal{V}'} = S_{\mathcal{V}}$.
Given a point $\lambda = (\lambda_1,\lambda_2)$ on the curve associated to $\mathcal{V}$ and $\mathcal{V}'$,
let us choose a direction $\xi = (\xi_1,\xi_2)$ such that 
\[
\xi_1(\lambda_1I-A_1) + \xi_2(\lambda_2I-A_2), \quad \xi_1(\lambda_1I-A_1') + \xi_2(\lambda_2I-A_2')
\]
are both invertible.
Notice that
\begin{eqnarray}\label{eqn:invWithT}
T \cdot \left(\xi_1(\lambda_1I-A_1) + \xi_2(\lambda_2I-A_2)\right) = 
\xi_1(\lambda_1T-TA_1) + \xi_2(\lambda_2T-TA_2) = \nonumber\\
= \xi_1(\lambda_1T-A_1'T) + \xi_2(\lambda_2T-A_2'T) =
\left(\xi_1(\lambda_1I-A_1') + \xi_2(\lambda_2I-A_2')\right) \cdot T.\nonumber
\end{eqnarray}
This implies that
\begin{equation}
\left(\xi_1(\lambda_1I-A_1') + \xi_2(\lambda_2I-A_2')\right)^{-1} \cdot T  = T \cdot \left(\xi_1(\lambda_1I-A_1) + \xi_2(\lambda_2I-A_2)\right)^{-1}.
\end{equation}
The transfer function of $\mathcal{V}'$ is given by:
\[
D' + C'T \cdot \left(\xi_1(\lambda_1I-A_1) + \xi_2(\lambda_2I-A_2)\right)^{-1} \cdot \tilde{B}(\xi_1\sigma_1+\xi_2\sigma_2)
\]
By (\ref{eqn:invWithT}), this is equal to
\begin{eqnarray}
D' + C' \cdot \left(\xi_1(\lambda_1I-A_1') + \xi_2(\lambda_2I-A_2')\right)^{-1} \cdot T \tilde{B}(\xi_1\sigma_1+\xi_2\sigma_2) = \nonumber\\
= D' + C' \cdot \left(\xi_1(\lambda_1I-A_1') + \xi_2(\lambda_2I-A_2')\right)^{-1} \cdot \tilde{B}'(\xi_1\sigma_1+\xi_2\sigma_2) \nonumber
\end{eqnarray}
and the latter is exactly the transfer function of $\mathcal{V}$.
\end{proof}

The following result is an elaboration of (\ref{fact:classical}).

\begin{prop}\label{prop:classicalReal}
Let $\Phi(\lambda) = D+C(\lambda I-A)^{-1}B$ be a minimal realization of a rational matrix function,
where $A:\mathcal{H} \to \mathcal{H}$, $B:\mathcal{E} \to \mathcal{H}$.
Let $\Psi(\lambda) = D'+C'(\lambda I - A')^{-1}B'$ be a minimal realization of another rational matrix function,
where $A':\mathcal{H}' \to \mathcal{H}'$, $B':\mathcal{E} \to \mathcal{H}'$.
Suppose $D,D'$ are invertible.
If $\opn{LP}(\Psi) \subseteq \opn{LP}(\Phi)$ then there exists a surjective linear map $T:\mathcal{H} \to \mathcal{H}'$,
such that $TA = A'T$, and $TB = B'$. 
\end{prop}
\begin{proof}
By multiplying $\Phi$ and $\Psi$ from the left by $D^{-1}$ and $D'^{-1}$ if necessary,
we may assume without loss of generality that $D = I$ and $D' = I$.
We will show that there is an injective linear map $T^{*}:(\mathcal{H}')^* \to (\mathcal{H})^*$\footnote{We remark that in this proof the star denotes the transpose or the adjoint of an operator rather than the hermitian adjoint.} such that
\[
T^{*} A'^{*} = A^{*}T^{*}, \quad B'^{*} = B^{*} T^{*}.
\]
Then, the corresponding surjective map $T:\mathcal{H} \to \mathcal{H}'$ will be the required surjection.
We do this in two steps.
\begin{steps}[wide, labelwidth=!, labelindent=0pt]
\item Assume first that both $\Phi$ and $\Psi$ have the property that they do not have a pole and a zero at the same point. For $\Phi$ this means that the spectrum of $A$ and the spectrum of $A^{\times} = A-BC$ do not intersect. Similarly, the claim about $\Psi$ is that the spectrum of $A'$ and the spectrum of $A'^{\times} = A'-B'C'$ do not intersect.
Under this assumption, 
according to (a left pole version of) \cite[Corollary 12.3.2]{BGR}\footnote{in the notation of \cite{BGR}, this result is applied to the set $\sigma$ equal to the spectrum of $A$; notice that the assumption  $\opn{LP}(\Psi) \subseteq \opn{LP}(\Phi)$ implies that 
the corresponding null-pole subspace for $\Psi$ is contained in that for $\Phi$.},
for any $x' \in (\mathcal{H}')^*$, 
there is a unique $x \in (\mathcal{H})^*$ such that
\[
B^{\prime *} (\lambda I - A')^{* -1} x' = B^* (\lambda I - A)^{* -1} x
\]
Setting $T^*(x') = x$ gives us the required injective linear map.
\item In the general case where $\Phi$ or $\Psi$ might have poles and zeroes at the same point,
the result \cite[Corollary 12.3.2]{BGR} mentioned above is more difficult.
To avoid analyzing this more complicated situation, 
we may instead use the fact that minimality implies that the pairs $(A,B)$ and $(A',B')$ are controllable. Hence, by the classical spectrum assignment problem,
we may replace $C$ by some $\widetilde{C}$, 
such that the spectrum of $A$ and the spectrum of $A-B\widetilde{C}$ are disjoint,
and the realization $\widetilde{\Phi}(\lambda) = D+\widetilde{C}(\lambda I-A)^{-1}B$ is still minimal.
Similarly,  we replace $C'$ by some $\widetilde{C}'$, such that the spectrum of $A'$ and the spectrum of $A'-B'\widetilde{C}'$ are disjoint, and the realization $\widetilde{\Psi}(\lambda) = D'+\widetilde{C}'(\lambda I - A')^{-1}B'$ is minimal. 
Since we did not replace $(A,B)$ and $(A',B')$, 
and since these realizations of $\widetilde{\Phi}$ and $\widetilde{\Psi}$ are minimal, 
by \cite[Theorem 4.2.1(iii)]{BGR},
we have that
\[
\opn{LP}(\widetilde{\Psi}) = \opn{LP}(\Psi) \subseteq \opn{LP}(\Phi) = \opn{LP}(\widetilde{\Phi}).
\]
It follows that we may apply Step 1 to these realizations of $\widetilde{\Phi}, \widetilde{\Psi}$,
and deduce the existence of the required injective linear map $T^{*}:(\mathcal{H}')^* \to (\mathcal{H})^*$.
\end{steps}
\end{proof}

Before the next lemma, let us recall the \textbf{restoration formula} for operator vessels,
see \cite[Equation (2-18)]{Vi} and \cite[Section 10.3]{LKMV} for more details.

\begin{fact}
Let $\mathcal{B}$ be an operator vessel, and let $\xi = (\xi_1,\xi_2) \in \mathbb{C}^2$ be a regular direction.
Suppose the algebraic curve $\mathbf{C}$ associated to $\mathcal{B}$ is of degree $m$.
Then, generically, given $\lambda \in \mathbb{C}$,
the line $\xi_1\cdot x + \xi_2\cdot y = \lambda$ intersects the curve $\mathbf{C}$ in $m$ distinct points.
Let us denote them by $(\lambda_1^1,\lambda_2^1),\dots, (\lambda_1^m,\lambda_2^m)$.
Then it holds that there is a (nonorthogonal) direct sum decomposition
\[
{\mathcal E} = E_{\opn{in}}(\lambda_1^1,\lambda_2^1) \dotplus \dots  \dotplus E_{\opn{in}}(\lambda_1^m,\lambda_2^m). 
\]
The restoration formula then says that
\begin{equation}\label{eqn:rest-formula}
S_{\xi}(\lambda) = \sum_{i=1}^m S_{\mathcal{B}}(\lambda_1^i,\lambda_2^i) \cdot P(\xi_1,\xi_2,\lambda_1^i,\lambda_2^i),
\end{equation}
where we denoted by $P(\xi_1,\xi_2,\lambda_1^i,\lambda_2^i)$ the corresponding projection operators,
which are holomorphic at all points on the curve $\mathbf C$ where the line $\xi_1\cdot x + \xi_2\cdot y = \lambda$
intersects $\mathbf{C}$ at $m$ distinct points.
\end{fact}

\begin{lem}\label{lem:directRegular}
Let 
\[
\mathcal{B} = (A_1,A_2,\tilde{B},C,D,\tilde{D},\sigma_1,\sigma_2,\gamma,\sigma_{1*},\sigma_{2*},\gamma_*;\mathcal{H},\mathcal{E},\tilde{\mathcal{E}},\mathcal{E}_*,\tilde{\mathcal{E}_*})
\]
and
\[
\mathcal{V} = (A_1',A_2',\tilde{B}',C',D',\tilde{D}',\sigma_1,\sigma_2,\gamma,\sigma_{1*}',\sigma_{2*}',\gamma_*';\mathcal{H}',\mathcal{E},\tilde{\mathcal{E}},\mathcal{E}_*',\tilde{\mathcal{E}_*}')
\]
be two operator vessels which share the same input bundle,
and set $S = S_{\mathcal{V}}, T = S_{\mathcal{B}}$.
Assume further that $\mathcal{B}$ is controllable.
Suppose that $\opn{LP}(S) \subseteq \opn{LP}(T)$.
Then there exists a direction $\xi = (\xi_1,\xi_2)$ which is a regular direction for both vessels,
and such that $\opn{LP}(S_{\xi}) \subseteq \opn{LP}(T_{\xi})$,
where 
\[
S_{\xi}(\lambda) = D+C(\lambda I - (\xi_1A_1 +\xi_2A_2))^{-1}\tilde{B}(\xi_1\sigma_1+\xi_2\sigma_2),
\]
and
\[
T_{\xi}(\lambda) = D'+C'(\lambda I - (\xi_1A_1' +\xi_2A_2'))^{-1}\tilde{B}'(\xi_1\sigma_1+\xi_2\sigma_2).
\]
\end{lem}
\begin{proof}
Let $Y \subseteq \mathbf{C}$ be the union of the joint spectrum of $(A_1,A_2)$ and the joint spectrum of $(A_1',A_2')$.
This is a finite set of points on the curve $\mathbf{C}$.
Take a direction $\xi = (\xi_1,\xi_2)$ such that $\xi$ is a regular direction for $\mathcal{B}$ and $\mathcal{V}$,
and such that for each $(\lambda_1,\lambda_2) \in Y$,
the line $\xi_1\cdot x + \xi_2 \cdot y = \xi_1 \cdot \lambda_1 + \xi_2\cdot \lambda_2$
will intersect $\mathbf{C}$ in $m$ distinct points,
and such that each of these lines contains only one point of $Y$.
We will show that $\opn{LP}(S_{\xi}) \subseteq \opn{LP}(T_{\xi})$.
Assume $S_{\xi}$ has a left pole of order $n$ at the point $z_0$ at direction $\phi$.
Then $z_0 \in \opn{Spec}(\xi_1 A_1 + \xi_2 A_2)$.
Since $A_1$ and $A_2$ commute, 
this implies that $z_0 = \xi_1 \lambda^0_1 + \xi_2 \lambda^0_2$ 
for some $(\lambda_1^0,\lambda^0_2) \in \opn{Spec}(A_1,A_2)$,
and in particular, by (\ref{fact:CH}), 
we have that $(\lambda^0_1,\lambda^0_2) \in \mathbf{C}$.
Using (\ref{eqn:rest-formula}), we may write
\[
S_{\xi}(z) = \sum_{i=1}^m S(\lambda_1^i(z),\lambda_2^i(z)) \cdot P(\xi_1,\xi_2,\lambda_1^i(z),\lambda_2^i(z))
\]
in a neighborhood of $z = z_0$.
Our above choice of $(\xi_1,\xi_2)$ 
guarantees that all summands in this sum except one are holomorphic in a neighborhood of $z_0$. 
Hence, up to equivalence of left pole datum, we may assume that
there is some $1 \le j \le n$, such that
$S(\lambda^j(z)) \cdot P(\xi_1,\xi_2,\lambda^j(z))$ has a left pole of order $n$ at the point $z_0$ at direction $\phi$.
Since $\opn{LP}(S) \subseteq \opn{LP}(T)$, 
we have that $\opn{LP}(S\circ P) \subseteq \opn{LP}(T \circ P)$.
The result now follows from applying (\ref{eqn:rest-formula}) to $T_{\xi}$.
\end{proof}

\begin{thm}\label{thm:Interpolation}
Let
\[
\mathcal{B} = (A_1,A_2,\widetilde{B},C,D,\widetilde{D},\sigma_1,\sigma_2,\gamma,\sigma_{1*},\sigma_{2*},\gamma_*;\mathcal{H},\mathcal{E},\widetilde{\mathcal{E}},\mathcal{E}_*,\widetilde{\mathcal{E}_*})
\]
be a minimal operator vessel, and let $S = S_{\mathcal{B}}:E_{\opn{in}}\to E_{\opn{out}}$ be its transfer function. 
Let $R:E_{\opn{in}}\to E_{\opn{in}}$ be a meromorphic bundle map, such that $R_{|L_\infty} = 1$. 
Then 
\[
\opn{LP}(R)\subseteq \opn{LP}(S)
\]
if and only if the following holds:
there exists an operator $K:\mathcal{H}\to \mathcal{E}$ such that the collection 
\[
\mathcal{V} = (A_1,A_2,\widetilde{B},K,I,I,\sigma_1,\sigma_2,\gamma,\sigma_1,\sigma_2,\gamma;\mathcal{H},\mathcal{E},\widetilde{\mathcal{E}},\mathcal{E},\widetilde{\mathcal{E}})
\]
is an operator vessel such that $R=S_{\mathcal{V}}$.
\end{thm}
\begin{proof}
Assume first that $\opn{LP}(R)\subseteq \opn{LP}(S)$.
Let us choose a minimal vessel
\[
\mathcal{V}' = (A_1',A_2',\tilde{B}',C',I,I,\sigma_1,\sigma_2,\gamma,\sigma_1,\sigma_2,\gamma;\mathcal{H}',\mathcal{E},\tilde{\mathcal{E}},\mathcal{E}_*,\tilde{\mathcal{E}_*})
\]
such that $S_{\mathcal{V}'} = R$.
By Lemma \ref{lem:directRegular}, 
let us choose a regular direction $\xi$ such that $\opn{LP}(R_{\xi})\subseteq \opn{LP}(S_{\xi})$.
According to Proposition \ref{prop:MinimalRealization} the realizations of $R_{\xi}$ and $S_{\xi}$ obtained 
by extending the vessels $\mathcal{V}'$ and $\mathcal{B}$ along the direction $\xi$ are both minimal.
By Proposition \ref{prop:classicalReal}, 
we deduce that there exists a surjective linear map $T:\mathcal{H} \to \mathcal{H}'$ such that
$TA_{\xi} = A'_{\xi}T$ and $T\tilde{B}_{\xi} = \tilde{B}'_{\xi}$.
By canceling the invertible matrix $\sigma_{\xi}$ from the right, 
the latter implies that $T\tilde{B} = \tilde{B}'$.
We may now apply Lemma \ref{lem:TtoVess},
and deduce that $TA_1 = A_1'$ and $TA_2= A_2'$.
This in turn implies by Lemma \ref{lem:newVessel} that
\[
\mathcal{V} = (A_1,A_2,\tilde{B},C'T,I,I,\sigma_1,\sigma_2,\gamma,\sigma_1,\sigma_2,\gamma;\mathcal{H},\mathcal{E},\tilde{\mathcal{E}},\mathcal{E}_*,\tilde{\mathcal{E}_*})
\]
is an operator vessel which satisfies $S_{\mathcal{V}} = R$, as claimed.

Conversely, suppose we may realize $R$ as $R=S_{\mathcal{V}}$, where 
\[
\mathcal{V} = (A_1,A_2,\widetilde{B},K,I,I,\sigma_1,\sigma_2,\gamma,\sigma_1,\sigma_2,\gamma;\mathcal{H},\mathcal{E},\widetilde{\mathcal{E}},\mathcal{E},\widetilde{\mathcal{E}})
\]
is an operator vessel. 
Let $\xi = (\xi_1,\xi_2)$ be a direction that is regular for both $\mathcal{B}$ and $\mathcal{V}$.
Then we may write:
\[
R_{\xi}(\lambda) = I+K(\lambda I - A_{\xi})^{-1}\widetilde{B}_{\xi}
\]
and
\[
S_{\xi}(\lambda) = D+C(\lambda I - A_{\xi})^{-1}\widetilde{B}_{\xi}.
\]
It follows that $(A_{\xi},\widetilde{B}_{\xi})$ is a global left pole pair for both $R_{\xi}$ and $S_{\xi}$.
Moreover, by Proposition \ref{prop:MinimalRealization}, 
the above realization of $S_{\xi}$ is minimal.
Then, as follows from (\ref{fact:classical}), 
we see that 
\[
\opn{LP}(R_{\xi})\subseteq \opn{LP}(S_{\xi}),
\]
and since $R$ and $S$ are obtained by restrictions of $R_{\xi}$ and $S_{\xi}$,
we deduce that 
\[
\opn{LP}(R)\subseteq \opn{LP}(S).
\]
\end{proof}

\subsection{The pole placement theorem}

\begin{lem}\label{lem:DivisorDiff}
Let $X$ be a compact Riemann surface, and let $\pi_E:E\to X, \pi_F:F\to X$ be two holomorphic vector bundles over $X$. Let $S,T:E\to F$ be two meromorphic bundle maps, and let $R = T^{-1} \circ S:E \to E$. Then $\opn{LP}(R)\subseteq \opn{LP}(S)$ if and only if $\opn{LZ}(T) \subseteq \opn{LZ}(S)$.
\end{lem}

Note that the condition $\opn{LP}(R)\subseteq \opn{LP}(S)$ makes sense,
because by (\ref{fact:whereZerosLive}), elements of each of these sets are triples 
$(\phi,n,z_0)$ where $\phi$ is a germ of an holomorphic section of $E^*$.

\begin{proof}
As this is a local question, 
we may assume that $X=\mathbb{C}$, 
and that $E,F$ are trivial. 
Suppose $\opn{LP}(R)\subseteq \opn{LP}(S)$. 
Assume $T$ has a left zero of order $k$ at direction $\phi$ at the origin. 
Let $\psi \in \mathcal{O}^{\times}_0$, 
such that $\phi(z)T(z) = z^k\psi(z)$. 
Hence, $\phi(z)S(z) = \phi(z) T(z)R(z) = z^k \psi(z) R(z)$. 
Assume $\psi(z)R(z) = z^l \alpha(z)$, where $\alpha \in \mathcal{O}^{\times}_0$.
If $l\ge 0$, 
then  $S$ has a zero in direction $\phi$ of order greater or equal to $k$, 
which proves the claim. 
Otherwise, if $l<0$, then setting $m=-l$, we see that $R$ has a left pole of order $m$ at direction $\alpha$. 
Hence, since $\opn{LP}(R)\subseteq \opn{LP}(S)$, 
it follows that for some $n\ge m$, there is a local section $\beta$ near $0$, such that $\beta(z)S(z) = z^{-n} \alpha(z)$. 
On the other hand, the above calculation shows that $\phi(z)S(z) = z^{k-m}\alpha(z)$. 
Since $m-k<n$, we get a contradiction, so $l\ge 0$, and the claim follows. The converse is proved similarly.
\end{proof}

Here is the main result of this paper.

\begin{thm}\label{thm:main}
Let $\mathcal{B}$ be a minimal operator vessel. Let $T:E_{\opn{in}}\to E_{\opn{out}}$ be a meromorphic bundle map whose poles do not lie over the singularities of $\mathbf{C}$. Then there is an admissible state feedback operator  $F:\mathcal{H}\to \mathcal{E}$ such that $T$ is the transfer function of the closed loop system $\ClosedLoop$ if and only if $\opn{LZ}(T)\subseteq \opn{LZ}(S_{\mathcal{B}})$ and $T_{|_{L_\infty}}={S_{\mathcal{B}}}_{|_{L_\infty}}$.
\end{thm}
\begin{proof}
Suppose first that $\opn{LZ}(T)\subseteq \opn{LZ}(S_{\mathcal{B}})$ and $T_{|_{L_\infty}}={S_{\mathcal{B}}}_{|_{L_\infty}}$. 
Let 
\[
R = T^{-1}\circ S_{\mathcal{B}} : E_{\opn{in}} \to E_{\opn{in}}.
\]
Clearly, $R_{|L_\infty} = 1$.
Moreover, by Lemma \ref{lem:DivisorDiff},
we have that $\opn{LP}(R)\subseteq \opn{LP}(S_{\mathcal{B}})$. 
Hence, by Theorem \ref{thm:Interpolation}, there is an operator 
$K:\mathcal{H}\to \mathcal{E}$ such that the collection
\[
\mathcal{V} = (A_1,A_2,\widetilde{B},K,I,I,\sigma_1,\sigma_2,\gamma,\sigma_1,\sigma_2,\gamma;\mathcal{H},\mathcal{E},\widetilde{\mathcal{E}},\mathcal{E},\widetilde{\mathcal{E}})
\]
is an operator vessel, and such that $R = S_{\mathcal{V}}$. 
Let $F = -K$. 
Note that by definition of the controller vessel, $\mathcal{V} = \Controller$.
In particular, by Proposition \ref{prop:Controller}, 
the operator $F$ is an admissible state feedback operator. 
Hence, by Proposition \ref{prop:Factorization}, we have that 
\[
S_{\ClosedLoop} = S_{\OpenLoop} \circ R^{-1},
\]
so that $T = S_{\ClosedLoop}$.
Conversely, suppose $T = S_{\ClosedLoop}$ for some admissible state feedback operator 
$F:\mathcal{H}\to \mathcal{E}$. 
Note that the controller vessel $\Controller$ is of the form required in Theorem \ref{thm:Interpolation}, 
hence the map $R = S_{\Controller}$ satisfies: $R_{|L_\infty} = 1$ and $\opn{LP}(R)\subseteq \opn{LP}(S_{\mathcal{B}})$. 
By Proposition \ref{prop:Factorization},
$T = S_{\mathcal{B}}\circ R^{-1}$. 
The result now follows from Lemma \ref{lem:DivisorDiff}.
\end{proof}

\section{Pole placement over line bundles}

\subsection{General theory}

In this final section, we analyze further the pole placement problem under the assumption that the vector bundles $E_{\opn{in}}$ and $E_{\opn{out}}$ are vector bundles of rank $1$, that is, line bundles. 
In this case, we will show below precisely how the geometry of the compact Riemann surface $X$ dictates the solution of the pole placement problem.
We further remark that many of the results of this section can be generalized to higher dimensional vector bundles, 
using the theory of matrix zero-pole divisors (see \cite{BCV2,BCV16,Ty,We}). 
For simplicity, we will however restrict ourselves to the line bundle case.

\begin{fact}
Recall that we assumed that the polynomials $\mathbf{p}_{0}(\lambda_1,\lambda_2)$ and $\mathbf{p}_{0*}(\lambda_1,\lambda_2)$ are given by 
\[
\mathbf{p}_{0}(\lambda_1,\lambda_2) = (\mathbf{f}_0(\lambda_1,\lambda_2))^r, \quad \mathbf{p}_{0*}(\lambda_1,\lambda_2) = (\mathbf{f}_{0*}(\lambda_1,\lambda_2))^s
\]
for some irreducible polynomials $\mathbf{f}_0, \mathbf{f}_{0*}$ and some $r, s\ge 1$. 
We shall now make the additional assumption (that holds generically) that $r = s =1$. 
This implies that $E_{\opn{in}}$ and $E_{\opn{out}}$ are line bundles over $X$,
so that the directional information of pole and zero data degenerate. 
In other words, under these assumptions, $\opn{LP}( S_{\OpenLoop} )$ and $\opn{LZ}( S_{\OpenLoop} )$ are ordinary effective divisors on $X$.
\end{fact}

\begin{fact}\label{fact:linebundleMap}
For any line bundle $E$ on $X$, recall that there is a natural isomorphism 
\[
\opn{Hom}_{\mathcal{O}_X}(E,E) \cong \mathcal{M}_X
\]
between the $\mathbb{C}$-algebra of meromorphic bundle maps $E\to E$ and the $\mathbb{C}$-algebra of meromorphic functions on $X$. 
\end{fact}

The above facts allow us to derive more explicit data from Theorem \ref{thm:main} in the line bundle case:

\begin{cor}
Let $\mathcal{B}$ be a minimal operator vessel as above, 
and assume that its input and output bundles are line bundles. 
Let $Z = \opn{LP}(S_{\OpenLoop}) \in \opn{Div}_{\ge 0}(X)$ be the left pole divisor of $S_{\OpenLoop}$.
Then the set of left divisors of closed loop systems with respect
to admissible state feedback operators 
$F:\mathcal{H}\to \mathcal{E}$ given by
\[
\{ \opn{\ell div}(S_{\ClosedLoop}) = \opn{LZ}(S_{\ClosedLoop})-\opn{LP}(S_{\ClosedLoop}) \mid F:\mathcal{H}\to \mathcal{E} \mbox{ is an admissible state feedback}\}
\]
is equal to the set of all divisors of the form 
\[
\opn{\ell div}(S_{\OpenLoop}) + \opn{div}(f) = \opn{LZ}(S_{\OpenLoop})-\opn{LP}(S_{\OpenLoop}) + \opn{div}(f)
\]
where $f \in \mathcal{M}_X$ is any meromorphic function on $X$ which satisfies:
\begin{enumerate}
\item The zero divisor of $f$ is contained in the left pole divisor of $S_{\OpenLoop}$,
that is $\opn{LZ}(f) \le Z$.
\item There is an equality $f_{|_{L_\infty}} = 1$.
\end{enumerate}
\end{cor}
\begin{proof}
Given such a function $f \in \mathcal{M}_X$, 
the isomorphism (\ref{fact:linebundleMap}) provides us 
with a meromorphic bundle map $R:E_{\opn{in}} \to E_{\opn{in}}$ with the same divisor data as $f^{-1}$, 
and with $R_{|_{L_\infty}} =1$. 
In particular, $\opn{LP}(R) = \opn{LZ}(R^{-1}) \subseteq \opn{LP}(S_{\mathcal{B}})$.
Letting $T = S_{\OpenLoop} \circ R^{-1}$, 
we know by Lemma \ref{lem:DivisorDiff} that $\opn{LZ}(T) \subseteq \opn{LZ}(S_{\mathcal{B}})$,
while the assumption on $f$ at infinity implies that $T_{|_{L_\infty}}={S_{\mathcal{B}}}_{|_{L_\infty}}$.
Hence, by Theorem \ref{thm:main},
there exists an admissible state feedback operator
$F:\mathcal{H}\to \mathcal{E}$
such that $S_{\ClosedLoop} = T$,
and there is an equality
\[
\opn{\ell div}(T) = \opn{\ell div}(S_{\OpenLoop}) - \opn{\ell div}(R) = \opn{\ell div}(S_{\OpenLoop}) + \opn{div}(f).
\]
Conversely, given an admissible state feedback operator $F$, 
let $T = S_{\ClosedLoop}$, 
and let $R = T^{-1} \circ S_{\mathcal{B}}$. 
Then using Theorem \ref{thm:main},
one sees that $R : E_{\opn{in}} \to E_{\opn{in}}$,
and that the meromorphic function $f^{-1}$ corresponding to it satisfies $f_{|_{L_\infty}} = 1$,
and Lemma \ref{lem:DivisorDiff} implies that $\opn{LZ}(f) \le Z$.
\end{proof}

It follows that under the assumption that the input and output bundles are line bundles, 
the pole placement problem essentially reduces to a classical interpolation problem over a compact Riemann surface $X$.
If we further concentrate on the case where we only attempt to place poles at points outside of the left zero divisor,
we may reduce the pole placement problem to the following:

\begin{prob}\label{theProb}
Let $\mathbf{C}$ be a projective plane algebraic curve of degree $m$ over $\mathbb{C}$ such that its intersection with the line at infinity contains $m$ different points. Let $X$ be the compact Riemann surface associated to its normalization, and let $Z$ be an effective divisor on $X$. For which effective divisors $P$ on $X$, there is a meromorphic function $f \in \mathcal{M}_X$ such that $\opn{div}(f) = Z - P$, and such that $f_{|_{L_\infty}} = 1$?
\end{prob}

\begin{fact}
As usual in the Riemann-Roch formalism, 
given a divisor $D$ over $X$, 
we denote by $L(D)$ the vector space 
\[
L(D) = \{ f\in \mathcal{M}_X : \opn{div}(f) \ge -D \} \cup \{0\}.
\]
We also set 
\[
\ell(D) = \dim_{\mathbb{C}}\left(L(D)\right).
\]
Denote by $D_{L_\infty}$ the effective divisor of the points of $X$ over the points of $\mathbf{C}$ at infinity.
\end{fact}

\begin{fact}\label{fact:parametrize}
Using the Riemann-Roch formalism, 
we may parametrize the functions $f \in \mathcal{M}_X$ that appear in Problem \ref{theProb} as follows:
given $g \in L(Z-D_{L_\infty})$,
by definition we have that $g(x) = 0$ for all $x \in L_{\infty}$,
and the pole divisor of $g$ is contained in $Z$.
Hence, the function $f = \frac{1}{g+1}$ satisfies that 
$f(x) = 1$ for all $x \in L_{\infty}$,
and the zero divisor of $f$ is contained in $Z$.
Conversely, if $f_{|_{L_\infty}} = 1$,
and the zero divisor of $f$ is contained in $Z$,
then $g = \frac{1}{f}-1 \in L(Z-D_{L_\infty})$.
\end{fact}

Our minimality assumption on the vessel $\mathcal{B}$ implies that $\deg(Z) = \dim_{\mathbb{C}}(\mathcal{H})$.
This is the content of the next lemma.

\begin{lem}
Let $\mathcal{B}$ be a minimal operator vessel as above, 
and assume that its input and output bundles are line bundles. 
Then the degree of its left pole divisor $Z$ satisfies
$\deg(Z) = \dim_{\mathbb{C}}(\mathcal{H})$.
\end{lem}
\begin{proof}
It is clear that $\deg Z \leq \dim_{{\mathbb C}}{\mathcal H}$.
To see the converse,
let $(\lambda_1^0,\lambda_2^0)$ be in the joint spectrum of $A_1$ and $A_2$.
If $v \in {\mathcal H}$, $v \neq 0$, is a joint eigenvector:
$A_1 \cdot v = \lambda_1^0 \cdot v$ and $A_2 \cdot v = \lambda_2^0 \cdot v$,
then $C \cdot v \neq 0$ 
(this is analogous to the classical Hautus test and follows from the observability of $\mathcal{B}$)
and also $C \cdot v \in E_{\text{out}}(\lambda_1^0,\lambda_2^0)$
because of vessel condition (A3) (see (\ref{eqn:vesselCond})).
Since $\dim E_{\text{out}}(\lambda_1^0,\lambda_2^0) = 1$,
it follows that the joint eigenspace is one-dimensional.
Therefore, 
$\xi_1 A_1 + \xi_2 A_2$ has a one-dimensional eigenspace 
with eigenvalue $\xi_1 \lambda_1^0 + \xi_2 \lambda_2^0$ for all 
but finitely many directions $(\xi_1,\xi_2)$, 
and thus a single Jordan cell of a certain size $n_0$ with this eigenvalue.
By \cite[Theorem 4.2.1]{BGR}, the restricted transfer function $S_\xi$
has a pole of order $n_0$ at $\xi_1 \lambda_1^0 + \xi_2 \lambda_2^0$,
and using the restoration formula as in the proof of Lemma \ref{lem:directRegular},
we conclude that $S_{\mathcal{B}}$ has a pole of order at least $n_0$ at $(\lambda_1^0,\lambda_2^0)$.
Since the sum of the sizes of Jordan cells corresponding to all the eigenvalues 
equals $\dim_{{\mathbb C}}{\mathcal H}$, we conclude that $\deg Z \geq \dim_{{\mathbb C}}{\mathcal H}$.
\end{proof}

\begin{fact}
Following this lemma, 
let us denote this number, the dimension of the state space, by $n$.
In the line bundle case we consider in this section,
we have that $\deg(D_{L_{\infty}}) = \dim_{\mathbb{C}}(\mathcal{E})$,
the dimension of the input space. Let us denote it by $m$.
Then $\deg(Z-D_{L_\infty}) = n-m$.
\end{fact}

\begin{fact}
Let $K$ be a canonical divisor of $X$. Thus, $K$ is the divisor of some meromorphic 1-form on $X$. Denote by $g$ the genus of $X$.
Applying the Riemann-Roch theorem to the divisor $Z - D_{L_{\infty}}$ implies that
\[
\ell(Z-D_{L_\infty}) = \ell(K-Z+D_{L_\infty}) + \deg(Z-D_{L_\infty}) - g +1 = \ell(K-Z+D_{L_\infty}) + n - m -g +1.
\]
Let us denote this number by $\fdim(\mathcal{B}) \in \mathbb{N}$.
We call this number the \textit{feedback dimension} of $\mathcal{B}$.
The above equality implies that
\[
\fdim(\mathcal{B}) \ge n -m -g+1 = \dim(\mathcal{H}) - \dim(\mathcal{E}) -g + 1.
\]
\end{fact}

\begin{fact}
Let $r = \fdim(\mathcal{B})$, 
and let $f_1,\dots,f_r$ be some basis of the vector space $L(Z-D_{L_\infty})$.
Denote by $X_0$ the non-compact Riemann surface $X \setminus L_{\infty}$,
obtained from $X$ by deleting all the points that lie over the line at infinity of $\mathbf{C}$.
Consider the complex manifold 
\[
\mathcal{Y} \subseteq \underbrace{X_0\times X_0 \times \dots \times X_0}_{\mbox{$r$ times}}
\]
given by all the $r$-tuples $\mathbf{p} = (p_1,\dots,p_r)$,
such that for $1 \le i < j \le r$, we have that $p_i \ne p_j$.
We define functions $\mathcal{M}:\mathcal{Y} \to M_r(\mathbb{C})$ 
and $\mathcal{P}:\mathcal{Y} \to \mathbb{C}$ as follows:
for each $\mathbf{p} = (p_1,\dots,p_r) \in \mathcal{Y}$,
define a square matrix $\mathcal{M}({\mathbf{p}}) \in M_r(\mathbb{C})$ by
$\mathcal{M}({\mathbf{p}}) = (b_{i,j})$, where $b_{i,j} = f_i(p_j)$,
and let $\mathcal{P}(\mathbf{p}) = \det(\mathcal{M}(\mathbf{p}))$.
\end{fact}

\begin{lem}
The function $\mathcal{P}:\mathcal{Y} \to \mathbb{C}$ is a meromorphic function.
Moreover, assuming that $r = \fdim(\mathcal{B}) > 0$, it is not identically zero.
\end{lem}
\begin{proof}
Since each of $f_1,\dots, f_r$ is a meromorphic function $X_0 \to \mathbb{C}$,
we see that $\mathcal{M}:\mathcal{Y} \to M_r(\mathbb{C})$ is also a meromorphic function,
and hence $\mathcal{P}$ is also a meromorphic function. 
Assume now that $r > 0$. 
Clearly, if $r = 1$ the claim holds.
Let us assume by induction that each minor of $\mathcal{M}$ is not identically zero.
Writing $\mathcal{P}$ as the Laplace expansion of $\mathcal{M}$ along the last column,
we may write
\[
\mathcal{P}(p_1,\dots,p_r) = \sum_{i=1}^r a_i(p_1,\dots,p_{r-1}) \cdot f_i(p_r)
\]
where each $a_i$ is a meromorphic function
\[
a_i:\underbrace{X_0\times X_0 \times \dots \times X_0}_{\mbox{$r-1$ times}} \to \mathbb{C}
\]
which is not identically zero. 
Let us choose some 
\[
(p_1,\dots,p_{r-1}) \in \underbrace{X_0\times X_0 \times \dots \times X_0}_{\mbox{$r-1$ times}}
\]
where for $1 \le i < j \le r-1$, we have that $p_i \ne p_j$,
such that $a_1(p_1,\dots,p_{r-1}) \ne 0$.
Then the fact that $f_1, \dots, f_r$ are linearly independent implies that for infinitely many $p \in X_0$, it holds that
\[
\sum_{i=1}^r a_i(p_1,\dots,p_{r-1}) \cdot f_i(p) \ne 0.
\]
Hence, the $r$-tuple $(p_1,\dots,p_{r-1},p)$ satisfies
$\mathcal{P}(p_1,\dots,p_{r-1},p) \ne 0$, as claimed.
\end{proof}

\begin{fact}\label{fact:NF}
Let us set
\[
\mathcal{NF} = \{(p_1,\dots,p_r) \in \mathcal{Y} \mid \mathcal{P}(p_1,\dots,p_r) = 0\}
\]
This set, the No-Feedback set, is a codimension $1$ subset of the $r$-dimensional complex manifold $\mathcal{Y}$.
In particular it is of measure $0$.
Note that this set is independent of the chosen basis $f_1,\dots,f_r$ of the vector space $L(Z-D_{L_\infty})$.
Note that $(p_1,\dots,p_r) \notin \mathcal{NF}$ if and only if
given $f \in L(Z-D_{L_\infty})$ such that $f(p_1) = f(p_2) = \dots = f(p_r) = 0$ it holds that $f \equiv 0$.
Hence, $(p_1,\dots,p_r) \notin \mathcal{NF}$ if and only if
\begin{equation}\label{eqn:RRofNF}
\ell(Z-D_{L_\infty}-\sum_{i=1}^r p_i) = 0.
\end{equation}
\end{fact}

\begin{fact}
Given an $r$-tuple $(p_1,\dots,p_r) \notin \mathcal{NF}$,
the fact that $\mathcal{P}(p_1,\dots,p_r) \ne 0$,
implies that there are $a_1,\dots,a_r \in \mathbb{C}$ such that
\begin{equation}\label{eqn:alleql1}
\sum_{i=1}^r a_i\cdot f_i(p_j) = -1
\end{equation}
for all $1 \le j \le r$.
Hence, letting 
\[
g = \sum_{i=1}^r a_i\cdot f_i \in L(Z-D_{L_\infty}),
\]
and applying the construction of (\ref{fact:parametrize}) to $g$,
we obtain that $f = \frac{1}{g+1}$ is a meromorphic function on $X$,
which satisfies $f(x) = 1$ for all $x \in L_{\infty}$,
its zero divisor is contained in $Z$,
and $f$ has a pole in each of the points $p_1,\dots,p_r$.
Further, note that since for $(p_1,\dots,p_r) \notin \mathcal{NF}$ the matrix $\mathcal{M}(p_1,\dots,p_r)$ is invertible, 
the function $g$ from (\ref{eqn:alleql1}) is unique.
Hence, there is a unique meromorphic function $f$ on $X$
which satisfies $f(x) = 1$ for all $x \in L_{\infty}$,
its zero divisor is contained in $Z$,
and $f$ has a pole in each of the points $p_1,\dots,p_r$.
\end{fact}

To summarize the above discussion,
we have proved the following theorem:

\begin{thm}
Consider a minimal operator vessel
\[
\mathcal{B} = (A_1,A_2,\widetilde{B},C,D,\widetilde{D},\sigma_1,\sigma_2,\gamma,\sigma_{1*},\sigma_{2*},\gamma_*;\mathcal{H},\mathcal{E},\widetilde{\mathcal{E}},\mathcal{E}_*,\widetilde{\mathcal{E}_*})
\]
as above, and assume that the input and output bundles of $\mathcal{B}$ are line bundles over the compact Riemann surface $X$. 
Denote by $g$ the genus of $X$,
let $Z = \opn{LZ}(S_{\OpenLoop}) \in \opn{Div}(X)$,
and let $K \in \opn{Div}(X)$ be a canonical divisor on $X$.
Then for
\[
r = \fdim(\mathcal{B}) = \ell(K-Z+D_{L_\infty}) + \dim(\mathcal{H}) - \dim(\mathcal{E}) -g +1 \ge \dim(\mathcal{H}) - \dim(\mathcal{E}) -g +1,
\]
given any $r$ distinct points $p_1,\dots,p_r$ of $X_0$,
except possibly tuples belonging to a codimension $1$ subset of measure $0$
\[
\mathcal{NF} \subseteq \underbrace{X_0\times X_0 \times \dots \times X_0}_{\mbox{$r$ times}}
\]
there exists an admissible state feedback operator $F:\mathcal{H} \to \mathcal{E}$,
such that the closed loop system $\ClosedLoop$ has a pole at $p_1,\dots,p_r$.
Moreover, for all $(p_1,\dots,p_r) \notin \mathcal{NF}$,
the rest of the poles of $\ClosedLoop$ are determined uniquely by the choice of the $r$ poles $p_1,\dots,p_r$.
\end{thm}

We remark that in the theorem above, 
some of the poles $p_1,\dots,p_r$ could be canceled if they happen to coincide with the zeros of the open loop system.

\subsection{Examples of small genus}

\subsubsection{Genus $0$}

Assume that the compact Riemann surface $X$ is of genus $0$.
In this particular case, the Riemann-Roch theorem says that
\[
\fdim(\mathcal{B}) = \ell(Z-D_{L_\infty}) = \deg(Z-D_{L_\infty}) +1 = 
\dim(\mathcal{H}) - \dim(\mathcal{E}) + 1 = n - m +1.
\]
or $\fdim(\mathcal{B}) = 0$ if this number is $\le 0$.
Setting $r = \fdim(\mathcal{B})$, and choosing any $r$ distinct points $p_1,\dots, p_r$,
note that
\[
\ell(Z-D_{L_\infty} - \sum_{i=1}^r p_i) = \deg(Z-D_{L_\infty}) -r +1 = 0
\]
Hence, it follows from (\ref{eqn:RRofNF}) that in this case $\mathcal{NF} = \emptyset$.
Thus, the pole placement theorem says in this case that given $r = \dim(\mathcal{H}) - \dim(\mathcal{E}) + 1$ distinct points $p_1,\dots, p_r$ in $X$, 
there is a closed loop system obtained by state feedback such that $p_1,\dots,p_r$ are poles of this system. The rest of the poles of the closed loop system are then uniquely determined by the choice of these $r$ poles. 
Consider the further specialized case where $m = 1$,
so that $X$ is simply $\mathbb{P}^1\mathbb{C}$. 
Here, we may take $A_1 = A_2 = A$, $\sigma_1 = \sigma_2 = \sigma_{1*} = \sigma_{2*} = 1$,
$\gamma = \gamma_* = 0$, and $\tilde{B} = B$.
Under these assumptions, the vessel $\mathcal{B}$ represents a classical 1D continuous-time time-invariant linear system, and we recover the classical pole placement theorem:
for any choice of $n = \dim(\mathcal{H})$ points, 
one can construct a closed loop system whose poles are the prescribed points.

\subsubsection{Genus $1$}

Assume that the compact Riemann surface $X$ is of genus $1$.
Let us choose some specific point $c_\infty \in L_{\infty} \cap X$.
Using this choice, $X$ has the structure of an elliptic curve over $\mathbb{C}$.
In particular, its points have the structure of an abelian group,
where addition is defined so that $p_1 + p_2 + p_3 = 0$, 
if and only if $p_1,p_2,p_3 \in X$ are the intersection of $X$ and a line.
Moreover, the chosen pointy $c_{\infty}$ is the identity of this group.

Denote this group operation by $\oplus$,
and let $\Phi:\opn{Div}(X) \to (X,\oplus)$ be the canonical group homomorphism.
In this genus one case, the Riemann-Roch theorem states that
\[
\fdim(\mathcal{B}) = \ell(Z-D_{L_\infty}) = \deg(Z-D_{L_\infty}) = 
\dim(\mathcal{H}) - \dim(\mathcal{E})  = n - m.
\]
assuming this number is positive.
If this number is negative, then $\fdim(\mathcal{B}) = 0$.
If $n = m$, then this number is either $0$ or $1$.
It is $1$ if and only if $Z-D_{L_{\infty}}$ is a principal divisor, 
equivalently, if $\Phi(Z-D_{L_{\infty}}) = c_\infty$.

Let $r = \fdim(\mathcal{B})$, 
and suppose that $r > 0$ and that $n > m$.
Given $r$ distinct points $p_1,\dots, p_r$ in $X$,
we have that 
\[
\deg(Z-D_{L_\infty} - \sum_{i=1}^r p_i) = 0.
\]
Hence, we have 
\[
\ell(Z-D_{L_\infty} - \sum_{i=1}^r p_i) = 0
\]
if and only if $\Phi(Z-D_{L_\infty} - \sum_{i=1}^r p_i) \ne c_{\infty}$.
Using (\ref{eqn:RRofNF}), we obtain the following characterization of $\mathcal{NF}$:
for any $r-1$ distinct points $p_1,\dots,p_{r-1}$,
there is a unique point $p_r \in X_0$ such that
$(p_1,\dots,p_r) \in \mathcal{NF}$. 
This point is given by 
\begin{equation}\label{eqn:pr}
p_r = \Phi(Z-D_{L_\infty} - \sum_{i=1}^{r-1} p_i).
\end{equation}
In particular, $\mathcal{NF} \ne \emptyset$.

To summarize the genus 1 case, the pole placement theorem in this case states that given 
$r = \dim(\mathcal{H}) - \dim(\mathcal{E})$ distinct points $p_1,\dots,p_r$,
such that the point $p_r$ is not the unique point that satisfies (\ref{eqn:pr}),
there is a closed loop system obtained by state feedback such that $p_1,\dots,p_r$ are poles of this system, and the rest of the poles of the closed loop system are then uniquely determined by the choice of these $r$ poles.

\subsubsection{Higher genus}

Suppose now that $X$ is a compact Riemann surface of genus $g > 1$.
As the genus of $X$ increases, it becomes more difficult to make a precise analysis of $\fdim(\mathcal{B})$ and of the set $\mathcal{NF}$. 
If however the dimension of the state space $\mathcal{H}$ is large enough compared to the dimension of the input space $\mathcal{E}$, we know the following from the Riemann-Roch theorem: assuming that 
\begin{equation}\label{eqn:morethangenus}
\dim(\mathcal{H}) - \dim(\mathcal{E}) > 2\cdot g - 2,
\end{equation}
there is an equality
\[
\fdim(\mathcal{B}) = \ell(Z-D_{L_\infty}) = \deg(Z-D_{L_\infty}) - g +1 = \dim(\mathcal{H}) - \dim(\mathcal{E}) -g +1.
\]
It follows, that, under the assumption (\ref{eqn:morethangenus}),
it is possible to place, generically, 
\[
\dim(\mathcal{H}) - \dim(\mathcal{E}) -g +1
\]
poles, except possibly if these points belong to the measure zero set $\mathcal{NF}$ introduced in (\ref{fact:NF}).

\textbf{Acknowledgments.}
Liran Shaul was partially supported by the Israel Science Foundation (grant no. 1346/15).
He was also partially supported by Charles University Research Centre program No.UNCE/SCI/022,
and by the grant GA~\v{C}R 20-13778S from the Czech Science Foundation.
Victor Vinnikov was partially supported by the Israel Science Foundation (grant no. 2123/17),
by the US–Israel Binational Science Foundation Grant No. 2010432,
and by the Deutsche Forschungsgemeinschaft (DFG) Grant No. SCHW 1723/1-1.
The authors thank the anonymous referee for many comments and suggestions that helped improving the manuscript.

\end{document}